\documentclass{amsart}
\usepackage{graphicx}
\usepackage{amsmath}
\usepackage{amssymb}
\usepackage{amsthm}
\usepackage{amsbsy}
\usepackage{enumitem}
\usepackage{bbm}
\usepackage[l2tabu,orthodox]{nag}
\usepackage[all,error]{onlyamsmath}
\usepackage[strict]{csquotes}
\usepackage{url}

\newtheorem{theorem}{Theorem}
\newtheorem{lemma}[theorem]{Lemma}
\newtheorem{corollary}[theorem]{Corollary}

\theoremstyle{definition}

\newtheorem{remark}[theorem]{Remark}
\newtheorem{example}[theorem]{Example}

\numberwithin{equation}{section}
\numberwithin{theorem}{section}

\renewcommand{\le}{\leqslant}
\renewcommand{\ge}{\geqslant}

\newcommand{\lin}{\mathrm{Lin}\,}

\setlist[enumerate,1]{label=(\roman*)}

\begin{document}
	\title{Unitarily equivalent bilateral weighted shifts with operator weights}
	\author{Michał Buchała}
	\begin{abstract}
		We study unitarily equivalent bilateral weighted shifts with operator weights. We establish a general characterization of unitary equivalence of such shifts under the assumption that the weights are quasi-invertible. We prove that under certain assumptions unitary equivalence of bilateral weighted shifts with operator weights defined on $ \mathbb{C}^{2} $ can always be given by a unitary operator with at most two non-zero diagonals. We provide examples of unitarily equivalent shifts with weights defined on $ \mathbb{C}^{k} $ such that every unitary operator, which intertwines them has at least $ k $ non-zero diagonals.
	\end{abstract}

	\maketitle


	
		
		
	\section{Introduction}
Classical weighted shifts (both unilateral and bilateral) have been studied extensively for many years (see \cite{Shi1} for comprehensive work on weighted shifts). One possible generalization of classical weighted shifts is to replace scalar weights with operator weights. We consider unilateral weighted shifts:
\begin{equation*}
    \ell^{2}(\mathbb{N},H)\ni (x_{i})_{i\in \mathbb{Z}}\mapsto (0,S_{1}x_{0},S_{2}x_{1},\ldots)\in \ell^{2}(\mathbb{N},H)
\end{equation*}
and bilateral weighted shifts:
\begin{equation*}
    \ell^{2}(\mathbb{Z},H)\ni (x_{i})_{i\in \mathbb{Z}}\mapsto (S_{i}x_{i-1})_{i\in \mathbb{Z}} \in \ell^{2}(\mathbb{Z},H), 
\end{equation*}
where all $ S_{i} $'s are bounded operators on a Hilbert space $ H $. The matrix representation of bilateral weighted shift takes the form
\begin{align}
    \label{FormMatrixOfOperatorShift}
    \begin{bmatrix}
        \ddots &  \ddots & \ddots & \ddots & \ddots\\
        \ddots & 0 & 0 & 0 & \ddots\\
        \ddots &  S_{0} & \boxed{0} & 0 & \ddots\\
        \ddots & 0 & S_{1} & 0 & \ddots\\
        \ddots & \ddots & \ddots & \ddots & \ddots
    \end{bmatrix},
\end{align}
where $ \boxed{\,\cdot\,} $ denotes the entry in zeroth column and zeroth row of this matrix.
In \cite{Lam3} Lambert studied unitary equivalence of unilateral weighted shifts with invertible weights. In \cite{Oro1} Orov\v canec characterized unitarily equivalent unilateral weighted shifts with quasi-invertible weights (see also \cite{AnaChaJabSto1} for the similar result for shifts with weights having dense ranges). In \cite{IvaOro1} the authors studied similar and quasi-similar unilateral weighted shifts and in \cite{Iva1} the similarity and quasi-similarity relation for bilateral weighted shifts was investigated. In \cite{Guy1} there was shown that if the weights can be divided into two sequences of normal and commuting operators, then the bilateral shift with such weights is unitarily equivalent to the shift with non-negative weights. Recently, in \cite{Kos1} Kośmider characterized bilateral shifts with quasi-invertible weights, which are unitarily equivalent by a unitary operator of diagonal form. The aim of this paper is to generalize the results of \cite{Kos1}. In Section \ref{SecPreliminaries} we introduce the notation and recall several basic facts about bilateral weighted shifts, which are crucial in the subsequent part of this paper. The main results are presented in Section \ref{SecUnitaryEquivalence}. In Theorem \ref{ThmGeneralCharacterizationOfUnitaryEquivalence} we present the counterpart of \cite[Corollary 3.3]{Lam3} for bilateral weighted shifts. In Corollary \ref{CorCharacterizationOfUnitaryEquivalenceByDiagonalForm} we show how to deduce the condition given in \cite[Corollary 2.4]{Kos1} from our result. We also give some more convenient characterizations of unitary equivalence under further assumptions on weights. In Theorem \ref{ThmCharacterizationOfUnitaryEquivalenceByPolarDecompositions} we characterize unitary equivalence in terms of factors in polar decomposition of the bilateral weighted shift. The remaining part of Section \ref{SecUnitaryEquivalence} is devoted to study certain shifts with positive weights. In Theorem \ref{Thm2DimUnitaryEquivalenceAtMostTwoNonZeroDiagonals} we prove that when $ H $ is two dimensional, then under certain assumptions on weights the unitary equivalence is always given by a unitary operator with at most two non-zero diagonals. At the end of this section we provide a class of examples of two unitarily equivalent shifts with weights on $ \mathbb{C}^{k} $ such that every unitary operator intertwining these two shifts has at least $ k $ non-zero diagonals.
	\section{Preliminaries}
\label{SecPreliminaries}
Denote by $ \mathbb{N} $ and $\mathbb{Z} $ the set of non-negative integers and integers, respectively and by $ \mathbb{R} $ and $ \mathbb{C} $ the field of real and complex numbers, respectively. For $ p \in \mathbb{R} $ and $ A\subset \mathbb{R} $ we set
\begin{equation*}
	A_{p} = \{x\in A\!: x\ge p \}.
\end{equation*}
If $ H,K $ are complex Hilbert spaces, then $ \mathbf{B}(H,K) $ stands for the space of all linear and bounded operators from $ H $ to $ K $; if $ H = K $, then we simply write $ \mathbf{B}(H) $. For $ T\in \mathbf{B}(H) $ we denote by $ \mathcal{R}(T) $ and $ \mathcal{N}(T) $ the range of $ T $ and the kernel of $ T $, respectively. An operator $ T\in \mathbf{B}(H) $ is called \textit{normal} if $ T $ commutes with its adjoint; $ T $ is \textit{positive} if $ \langle Th,h\rangle >0 $ for every $ h\in H $, $ h\not=0 $. We say that an operator $ T\in \mathbf{B}(H) $ is \textit{quasi-invertible} if $ \mathcal{N}(T) = \{0\} $ and $ \mathcal{N}(T^{\ast}) = \{0\} $. It can be easily seen that
\begin{equation}
    \label{FormQuasiInvertibleClosedUnderProducts}
    T,S \text{ -- quasi-invertible} \implies TS \text{ -- quasi-invertible}, \qquad T,S\in \mathbf{B}(H),
\end{equation}
and that
\begin{equation*}
    T \text{ -- positive} \implies T \text{ -- quasi-invertible}, \qquad T\in \mathbf{B}(H).
\end{equation*} 
An operator $ U\in \mathbf{B}(H) $ is called a \textit{partial isometry} if $ U|_{\mathcal{N}(U)^{\perp}} $ is isometric. If $ T\in \mathbf{B}(H) $, then there exists the unique partial isometry $ U\in \mathbf{B}(H) $ such that $ \mathcal{N}(T) = \mathcal{N}(U) $ and $ T = U\lvert T\rvert $, where $ \lvert T\rvert = (T^{\ast}T)^{1/2} $ (see \cite[Theorem 7.20]{Wei1}).
\begin{equation}
    \label{FormPolarDecompositionGeneral}
    \begin{minipage}{10cm}
        The polar decomposition of $T$ is the decomposition $T = U\lvert T\rvert$, where $U$ is the unique partial isometry satisfying $\mathcal{N}(U) = \mathcal{N}(T)$.
    \end{minipage}
\end{equation}
We say that the operators $ T,S\in \mathbf{B}(H) $ are \textit{unitarily equivalent} if there exists a unitary operator $ U\in \mathbf{B}(H) $ such that $ US = TU $. We can easily show that if $ T,S\in \mathbf{B}(H) $ are unitarily equivalent, then so are $ S^{\ast} $ and $ T^{\ast} $ (by the same unitary operator). A family $ \mathcal{F}\subset \mathbf{B}(H) $ of operators is \textit{doubly commuting} if for every $ A,B\in \mathcal{F} $, $ AB = BA $ and $ A^{\ast}B = BA^{\ast} $. If $ (M_{n})_{n\in \mathbb{Z}} $ is the sequence of subspaces of $ H $, then we denote by $ \bigvee_{n\in \mathbb{Z}} M_{n} $ the closed linear span of $ \bigcup_{n\in \mathbb{Z}} M_{n} $. If $ (H_{i})_{i\in \mathbb{Z}} $ is the sequence of Hilbert spaces, then we define its orthogonal sum as follows:
\begin{equation*}
    \bigoplus_{i\in \mathbb{Z}} H_{i} = \left\{ h = (h_{i})_{i\in \mathbb{Z}}\in \prod_{i\in \mathbb{Z}}H_{i}\!: \sum_{i\in \mathbb{Z}} \lVert h_{n}\rVert^{2}<\infty \right\};
\end{equation*}
this is a Hilbert space with the inner product given by the formula:
\begin{equation*}
    \langle (h_{i})_{i\in \mathbb{Z}}, (h_{i}')_{i\in \mathbb{Z}}\rangle = \sum_{i\in \mathbb{Z}}\langle h_{i},h_{i}'\rangle, \qquad (h_{i})_{i\in \mathbb{Z}},(h_{i}')_{i\in \mathbb{Z}}\in \bigoplus_{i\in \mathbb{Z}} H_{i}.
\end{equation*}
If $ H_{i} = H $ for every $ i\in \mathbb{Z} $, then the above orthogonal sum will be denoted by $ \ell^{2}(\mathbb{Z},H) $. 
For $ x\in H $ and $ k\in \mathbb{Z} $ denote by $ x^{(k)} \in \ell^{2}(\mathbb{Z},H) $ the vector given as follows: $ x^{(k)}_{k} = x $ and $ x^{(k)}_{i} = 0 $ for $ i\in \mathbb{Z}\setminus\{k\} $.
Note that every operator $ T\in \mathbf{B}(\ell^{2}(\mathbb{Z},H)) $ has a matrix representation $ [T_{i,j}]_{i,j\in \mathbb{Z}} $, where $ T_{i,j}\in \mathbf{B}(H) $, satisfying the following formula (see \cite[Chapter 8]{Hal1}):
\begin{equation*}
    T(x_{i})_{i\in \mathbb{Z}} = \left( \sum_{j\in \mathbb{Z}} T_{i,j} x_{j} \right)_{i\in \mathbb{Z}}\!\!\!\!\!\!\!\!\!\!, \qquad (x_{i})_{i\in \mathbb{Z}} \in \ell^{2}(\mathbb{Z},H).
\end{equation*}
Observe that
\begin{equation}
    \label{FormAdjointOfMatrix}
    T = [T_{i,j}]_{i,j\in \mathbb{Z}} \in \mathbf{B}(\ell^{2}(\mathbb{Z},H)) \implies T^{\ast} = [T^{\ast}_{j,i}]_{i,j\in \mathbb{Z}},
\end{equation}
and
\begin{equation}
    \label{FormProductOfMatrices}
    T = [T_{i,j}]_{i,j\in \mathbb{Z}}, S = [S_{i,j}]_{i,j\in \mathbb{Z}} \in \mathbf{B}(\ell^{2}(\mathbb{Z},H)) \implies TS = \left[ \sum_{k\in \mathbb{Z}} T_{i,k}S_{k,j} \right]_{i,j\in \mathbb{Z}}.
\end{equation} For the convenience we introduce the following notation: if $ (A_{i})_{i\in \mathbb{Z}}\subset \mathbf{B}(H) $ is a uniformly bounded sequence of operators, then by $ D[(A_{i})_{i\in \mathbb{Z}}] \in \mathbf{B}\left( \ell^{2}(\mathbb{Z},H) \right) $ we denote the diagonal operator with operators $ A_{i} $ ($ i\in \mathbb{Z} $) on the diagonal, that is, $ D[(A_{i})_{i\in \mathbb{Z}}]\!: \ell^{2}(\mathbb{Z},H) \to \ell^{2}(\mathbb{Z},H) $ is defined as
\begin{align}
    \label{FormDiagonalOperatorDef}
    D[(A_{i})_{i\in \mathbb{Z}}](x_{i})_{i\in \mathbb{Z}} = (A_{i}x_{i})_{i\in \mathbb{Z}}, \qquad (x_{i})_{i\in \mathbb{Z}} \in \ell^{2}(\mathbb{Z},H).
\end{align}
If $ (S_{i})_{i\in \mathbb{Z}}\subset \mathbf{B}(H) $ is a uniformly bounded sequence of operators, then we define the bilateral weighted shift $ S\in \mathbf{B}\left(\ell^{2}(\mathbb{Z},H) \right) $ with weights $ (S_{i})_{i\in \mathbb{Z}} $ by the formula:
\begin{equation*}
    S(x_{i})_{i\in \mathbb{Z}} = (S_{i}x_{i-1})_{i\in \mathbb{Z}}, \qquad (x_{i})_{i\in \mathbb{Z}}\in \ell^{2}(\mathbb{Z},H).
\end{equation*}
By $ F $ we denote the bilateral weighted shift with all weights equal to the identity operator on $ H $; it is easy to see that $ F $ is a unitary operator. We say that an operator $ T\in \mathbf{B}\left( \ell^{2}(\mathbb{Z},H) \right) $ is of diagonal form if there exists a uniformly bounded sequence $ (A_{i})_{i\in \mathbb{Z}}\subset \mathbf{B}(H) $ such that $ T = F^{k}D[(A_{i})_{i\in \mathbb{Z}}] $ for some $ k\in \mathbb{Z} $.
In the following lemma we gather basic properties of bilateral weighted shifts with operator weights; the proof (which is a straightforward application of \eqref{FormMatrixOfOperatorShift}, \eqref{FormAdjointOfMatrix} and \eqref{FormProductOfMatrices}) is left to the reader.
\begin{lemma}
    \label{LemBasicPropertiesOfOperatorShifts}
    Let $ H $ be a complex Hilbert space. Let $ (S_{i})_{i\in \mathbb{Z}}\subset \mathbf{B}(H) $ be a uniformly bounded sequence of operators and let $ S $ be the bilateral weighted shift $ S\in \mathbf{B}\left(\ell^{2}(\mathbb{Z},H) \right) $ with weights $ (S_{i})_{i\in \mathbb{Z}} $. Then:
    \begin{enumerate}
        \item for every $ n\in \mathbb{N}_{1} $,
        \begin{equation*}
            (S^{n})_{i,j} = \begin{cases}
                0, & \text{for } i\not=j+n\\
                S_{j+n}\cdots S_{j+1}, & \text{for } i = j+n
            \end{cases}, \qquad i,j\in \mathbb{Z},
        \end{equation*}
        are the entries of the matrix representation of $ S^{n} $,
        \item for every $ n\in \mathbb{N}_{1} $,
        \begin{equation*}
            (S^{\ast n})_{i,j} = \begin{cases}
                0, & \text{for } j\not=i+n\\
                S_{i+1}^{\ast}\cdots S_{i+n}^{\ast}, & \text{for } j = i+n
            \end{cases}, \qquad i,j\in \mathbb{Z},
        \end{equation*}
        are the entries of the matrix representation of $ S^{\ast n} $,
        \item for every $ n\in \mathbb{N}_{1} $,
        \begin{equation*}
            (S^{\ast n}S^{n})_{i,j} = \begin{cases}
                0, & \text{for } i\not=j\\
                \lvert S_{j+n}\cdots S_{j+1}\rvert^{2}, & \text{for } i = j
            \end{cases}, \qquad i,j\in \mathbb{Z},
        \end{equation*}
        are the entries of the matrix representation of $ S^{\ast n}S^{n} $,
        \item for every $ n\in \mathbb{N}_{1} $,
        \begin{equation*}
            (S^{n}S^{\ast n})_{i,j} = \begin{cases}
                0, & \text{for } j\not=i\\
                \lvert S_{i-n+1}^{\ast}\cdots S_{i}^{\ast}\rvert^{2}, & \text{for } j = i
            \end{cases}, \qquad i,j\in \mathbb{Z},
        \end{equation*}
        are the entries of the matrix representation of $ S^{n}S^{\ast n} $.
    \end{enumerate}
\end{lemma}
The next lemma characterizes the operators, which intertwines two bilateral weighted shifts (see also \cite[Lemma 4]{Pili1}).
\begin{lemma}
    \label{LemInterweavingWeightedShifts}
    Let $ H $ be a complex Hilbert space. Let $ (S_{i})_{i\in \mathbb{Z}}, (T_{i})_{i\in \mathbb{Z}} \subset \mathbf{B}(H) $ be two uniformly bounded sequences of operators. Denote by $ S,T $ the bilateral weighted shifts with weights $ (S_{i})_{i\in \mathbb{Z}} $ and $ (T_{i})_{i\in \mathbb{Z}} $, respectively. Suppose $ A = [A_{i,j}]_{i,j\in \mathbb{Z}}\in \mathbf{B}(\ell^{2}(\mathbb{Z},H)) $. Let $ n\in \mathbb{N}_{1} $. Then:
    \begin{enumerate}
        \item $ AS^{n} = T^{n}A $ if and only if
        \begin{equation*}
            A_{i+n,j+n}(S_{j+n}\cdots S_{j+1}) = T_{i+n}\cdots T_{i+1}A_{i,j}, \qquad i,j\in \mathbb{Z};
        \end{equation*}
        \item $ AS^{\ast n} = T^{\ast n}A $ if and only if
        \begin{equation*}
            A_{i,j}(S_{j+1}^{\ast}\cdots S_{j+n}^{\ast}) = T_{i+1}^{\ast}\cdots T_{i+n}^{\ast}A_{i+n,j+n}, \qquad i,j\in \mathbb{Z}
        \end{equation*}
    \end{enumerate}
\end{lemma}
\begin{proof}
    Observe that, by Lemma \ref{LemBasicPropertiesOfOperatorShifts},
    \begin{align*}
        (AS^{n})_{i,j} &= \sum_{k\in \mathbb{Z}} A_{i,k}(S^{n})_{k,j} = A_{i,j+n}S_{j+n}\cdots S_{j+1}, \qquad i,j\in \mathbb{Z},\\
        (T^{n}A)_{i,j} &= \sum_{k\in \mathbb{Z}} (T^{n})_{i,k}A_{k,j} = T_{i}\cdots T_{i-n+1}A_{i-n,j}, \qquad i,j\in \mathbb{Z},
    \end{align*}
    are the entires of matrix representations of $ AS^{n} $ and $ T^{n}A $, respectively. Replacing $ i $ with $ i+n $ in the above equalities, we obtain (i). In turn, again by Lemma \ref{LemBasicPropertiesOfOperatorShifts},
    \begin{align*}
        (AS^{\ast n})_{i,j} &= \sum_{k\in \mathbb{Z}} A_{i,k}(S^{\ast n})_{k,j} = A_{i,j-n}S_{j-n+1}^{\ast}\cdots S_{j}^{\ast}, \qquad i,j\in \mathbb{Z},\\
        (T^{\ast n}A)_{i,j} &= \sum_{k\in \mathbb{Z}} (T^{\ast n})_{i,k}A_{k,j} = T_{i+1}^{\ast}\cdots T_{i+n}^{\ast}A_{i+n,j}, \qquad i,j\in \mathbb{Z},
    \end{align*}
    are the entries of matrix representations of $ AS^{\ast n} $ and $ T^{\ast n}A $, respectively. Replacing $ j $ with $ j+n $ in the above equalities, we obtain (ii).
\end{proof}
In the following lemma we describe the polar decomposition of a bilateral weighted shift.
\begin{lemma}
    \label{LemPolarDecompositionOfShift}
    Let $ H $ be a complex Hilbert space. Let $ (S_{i})_{i\in \mathbb{Z}}\subset \mathbf{B}(H) $ be a uniformly bounded sequence of operators and let $ S $ be the bilateral weighted shift with weights $ (S_{i})_{i\in \mathbb{Z}} $. For $ i\in \mathbb{Z} $ let $ S_{i} = V_{i}\lvert S_{i}\rvert $ be the polar decomposition of $ S_{i} $. Denote by $ V $ the bilateral weighted shift with weights $ (V_{i})_{i\in \mathbb{Z}} $. Then $ S = V\lvert S\rvert $ is the polar decomposition of $ S $.
\end{lemma}
\begin{proof}
    For every $ (x_{i})_{i\in \mathbb{Z}}\in \ell^{2}(\mathbb{Z},H) $ we have
    \begin{align*}
        S(x_{i})_{i\in \mathbb{Z}} &= (S_{i}x_{i-1})_{i\in \mathbb{Z}} = (V_{i}\lvert S_{i}\rvert x_{i-1})_{i\in \mathbb{Z}} \\
        &= FD[(V_{i+1})_{i\in \mathbb{Z}}]D[(\lvert S_{i+1}\rvert)_{i\in \mathbb{Z}}](x_{i})_{i\in \mathbb{Z}}.
    \end{align*}
    By Lemma \ref{LemBasicPropertiesOfOperatorShifts}, $ D[(\lvert S_{i+1}\rvert)_{i\in \mathbb{Z}}] = \lvert S\rvert $. It can be easily verified that $ V = FD[(V_{i+1})_{i\in \mathbb{Z}}] $. Next,
    \begin{align*}
        \mathcal{N}(S) &= \left\{ (x_{i})_{i\in \mathbb{Z}}\in \ell^{2}(\mathbb{Z},H)\!: S_{i}x_{i-1} = 0 \text{ for all } i\in \mathbb{Z} \right\}\\
        &\stackrel{\eqref{FormPolarDecompositionGeneral}}{=}\left\{ (x_{i})_{i\in \mathbb{Z}}\in \ell^{2}(\mathbb{Z},H)\!: V_{i}x_{i-1} = 0 \text{ for all } i\in \mathbb{Z} \right\} = \mathcal{N}(V).
    \end{align*}
    Thus, $ V $ is a partial isometry satisfying $ \mathcal{N}(S) = \mathcal{N}(V) $, so $ S = V\lvert S\rvert $ is the polar decomposition of $ S $ as in \eqref{FormPolarDecompositionGeneral}.
\end{proof}
	\section{Unitary equivalence of weighted shifts}
\label{SecUnitaryEquivalence}
In this section we provide a general characterization of unitarily equivalent weighted shifts with operator weights. The result presented below is the counterpart of the theorem of Lambert (cf. \cite[Corollary 3.3]{Lam3}).
\begin{theorem}
    \label{ThmGeneralCharacterizationOfUnitaryEquivalence}
    Let $ H $ be a Hilbert space. Let $ (S_{i})_{i\in \mathbb{Z}}, (T_{i})_{i\in \mathbb{Z}} \subset \mathbf{B}(H) $ be two uniformly bounded sequences of quasi-invertible operators and let $ S,T $ be the bilateral weighted shifts with weights $ (S_{i})_{i\in \mathbb{Z}} $ and $ (T_{i})_{i\in \mathbb{Z}} $, respectively.
    \begin{enumerate}
        \item If $ US = TU $ for a unitary operator $ U\in \mathbf{B}(\ell^{2}(\mathbb{Z},H)) $, then for the isometry $ U_{0}\in \mathbf{B}(H,\ell^{2}(\mathbb{Z},H)) $ defined by $ U_{0}x = Ux^{(0)} $, $ x\in H $, the following conditions hold:
        \begin{align}
            \label{FormPositiveWeights}
            &U_{0}\lvert S_{n}\cdots S_{1}\rvert = D[(\lvert T_{i+n}\cdots T_{i+1}\rvert)_{i\in \mathbb{Z}}]U_{0}, \qquad n\in \mathbb{N}_{1},\\
            \label{FormNegativeWeights}
            &U_{0}\lvert S_{-n+1}^{\ast}\cdots S_{0}^{\ast}\rvert = D[(\lvert T_{i-n+1}^{\ast}\cdots T_{i}^{\ast}\rvert)_{i\in \mathbb{Z}}]U_{0}, \qquad n\in \mathbb{N}_{1},\\
            \label{FormWanderingProperty}
            & T^{[k]}(\mathcal{R}(U_{0}))\perp T^{[m]}(\mathcal{R}(U_{0})), \qquad k,m\in \mathbb{Z}, \ k\not= m,\\
            \label{FormClosedSpan}
            &\bigvee_{i\in \mathbb{Z}} T^{[i]}(\mathcal{R}(U_{0})) = \ell^{2}(\mathbb{Z},H),
        \end{align}
        where
        \begin{equation*}
            T^{[k]} = \begin{cases}
                T^{k}, & \text{if } k\in \mathbb{N}_{1}\\
                I, & \text{if }k = 0\\
                T^{\ast -k}, & \text{if } -k\in \mathbb{N}_{1}
            \end{cases}.
        \end{equation*}
        \item If there exists an isometry $ U_{0}\in \mathbf{B}(H,\ell^{2}(\mathbb{Z},H)) $ satisfying \eqref{FormPositiveWeights}-\eqref{FormClosedSpan}, then $ S $ and $ T $ are unitarily equivalent.
    \end{enumerate}
\end{theorem}
\begin{proof}
    (i). Let $ U = [U_{i,j}]_{i,j\in \mathbb{Z}}\in \mathbf{B} (\ell^{2}(\mathbb{Z},H)) $ be a unitary operator satisfying $ US = TU $. Then, we also have $ US^{\ast} = T^{\ast}U $. This implies that for every $ n\in \mathbb{N}_{1} $, $ U\lvert S^{n}\rvert^{2} = \lvert T^{n}\rvert^{2}U $ and $ U\lvert S^{\ast n}\rvert^{2} = \lvert T^{\ast n}\rvert^{2} U $. Using the square root theorem (see \cite[p.265]{Rie1}) and Berberian's trick (see \cite{Ber1}) we obtain that $ U\lvert S^{n}\rvert = \lvert T^{n}\rvert U $ and $ U\lvert S^{\ast n}\rvert = \lvert T^{\ast n}\rvert U $. By Lemma \ref{LemBasicPropertiesOfOperatorShifts}, the above equalities take the form
    \begin{align}
        \label{ProofFormModulusOfShift}
        U_{i,j}\lvert S_{j+n}\cdots S_{j+1}\rvert &= \lvert T_{i+n}\cdots T_{i+1}\rvert U_{i,j}, \qquad i,j \in \mathbb{Z},\ n\in \mathbb{N}_{1},\\
        \label{ProofFormModulusOfAdjoint}
        U_{i,j}\lvert S_{j-n+1}^{\ast}\cdots S_{j}^{\ast}\rvert &= \lvert T_{i-n+1}^{\ast}\cdots T_{i}^{\ast}\rvert U_{i,j}, \qquad i,j\in \mathbb{Z},\ n\in \mathbb{N}_{1}.
    \end{align}
    We will prove that the isometry $ U_{0} $ satisfies \eqref{FormPositiveWeights}-\eqref{FormClosedSpan}. By \eqref{ProofFormModulusOfShift} and \eqref{ProofFormModulusOfAdjoint}, \eqref{FormPositiveWeights} and \eqref{FormNegativeWeights} hold.  For $ n\in \mathbb{Z}\setminus\{0\} $ define $ U_{n}\!: H\to \ell^{2}(\mathbb{Z},H) $ by the formula:
    \begin{equation*}
        U_{n}x = Ux^{(n)}, \qquad x\in H.
    \end{equation*}
    Since $ U $ is unitary, $ U_{n} $ is isometric for all $ n\in \mathbb{Z} $. Next, by Lemma \ref{LemInterweavingWeightedShifts},
    \begin{equation}
        \label{ProofFormShiftInterweavingPositive}
        U_{i+n,n}S_{n}\cdots S_{1} = T_{i+n}\cdots T_{i+1}U_{i,0}, \qquad i\in \mathbb{Z},\ n\in\mathbb{N}_{1}.
    \end{equation}
    For $ n\in \mathbb{N}_{1} $ and $ i\in \mathbb{Z} $ let
    \begin{align}
        \label{ProofFormPolarDecompositionPositiveS}
        S_{n}\cdots S_{1} &= V_{n}\lvert S_{n}\cdots S_{1}\rvert\\
        \label{ProofFormPolarDecompositionPositiveT}
        T_{i+n}\cdots T_{i+1} &= W_{n,i}\lvert T_{i+n}\cdots T_{i+1}\rvert
    \end{align} 
    be the polar decompositions. Since, by \eqref{FormQuasiInvertibleClosedUnderProducts}, $ S_{n}\cdots S_{1} $ and $ T_{i+n}\cdots T_{i+1} $ are quasi-invertible, it follows from \eqref{FormPolarDecompositionGeneral} that partial isometries $ V_{n} $ and $ W_{n,i} $ ($ n\in \mathbb{N}_{1} $, $ i\in \mathbb{Z} $) are unitary operators. We have
    \begin{align*}
        &W_{n,i}\lvert T_{i+n}\cdots T_{i+1}\rvert U_{i,0} = T_{i+n}\cdots T_{i+1}U_{i,0}\\
        &\stackrel{\eqref{ProofFormShiftInterweavingPositive}}{=} U_{i+n,n}S_{n}\cdots S_{1} = U_{i+n,n}V_{n}\lvert S_{n}\cdots S_{1}\rvert, \qquad n\in \mathbb{N}_{1}, \ i\in \mathbb{Z}.
    \end{align*}
    This implies that for every $ n\in \mathbb{N}_{1}, \ i\in \mathbb{Z} $,
    \begin{equation*}
        W_{n,i}^{\ast}U_{i+n,n}V_{n}\lvert S_{n}\cdots S_{1}\rvert = \lvert T_{i+n}\cdots T_{i+1}\rvert U_{i,0} \stackrel{\eqref{ProofFormModulusOfShift}}{=} U_{i,0}\lvert S_{n}\cdots S_{1}\rvert.
    \end{equation*}
    Since $ S_{n}\cdots S_{1} $ is quasi-invertible, $ \mathcal{N}(\lvert S_{n}\cdots S_{1}\rvert) = \{0\} $, which implies that\linebreak $ \overline{\mathcal{R}(\lvert S_{n}\cdots S_{1}\rvert)} = H $. It follows from \eqref{ProofFormModulusOfShift} that $ W_{n,i}^{\ast}U_{i+n,n}V_{n} = U_{i,0} $; using the fact that $ W_{n,i} $ is unitary, it can be written equivalently as
    \begin{equation}
        \label{ProofFormEntriesOfUnitaryPositiveColumns}
        U_{i+n,n} = W_{n,i}U_{i,0}V_{n}^{\ast}, \qquad n\in \mathbb{N}_{1}, \ i\in \mathbb{Z}.
    \end{equation}
    From \eqref{ProofFormEntriesOfUnitaryPositiveColumns} and \eqref{FormDiagonalOperatorDef} we can derive that
    \begin{equation}
        \label{ProofFormEntriesOfUnitaryPositiveColumnsByZeroColumn}
        U_{n} = F^{n}D[(W_{n,i})_{i\in \mathbb{Z}}]U_{0}V_{n}^{\ast}, \qquad n\in \mathbb{N}_{1}.
    \end{equation}
    Now we prove similar formula for $ U_{-n} $ ($ n\in \mathbb{N}_{1} $); we provide the whole reasoning, because it differs from the above in some details. Again, by Lemma \ref{LemInterweavingWeightedShifts} we have
    \begin{equation}
        \label{ProofFormShiftInterweavingNegative}
        U_{i-n,-n}(S_{-n+1}^{\ast}\cdots S_{0}^{\ast}) = T_{i-n+1}^{\ast}\cdots T_{i}^{\ast}U_{i,0}, \qquad n\in \mathbb{N}_{1}, \ i\in \mathbb{Z}.
    \end{equation}
    For $ n\in \mathbb{N}_{1} $ and $ i\in \mathbb{Z} $ let
    \begin{align}
        \label{ProofFormPolarDecompositionNegativeS}
        S_{0}\cdots S_{-n+1} &= V_{-n}\lvert S_{0}\cdots S_{-n+1}\rvert\\
        \label{ProofFormPolarDecompositionNegativeT}
        T_{i}\cdots T_{i-n+1} &= W_{-n,i}\lvert T_{i}\cdots T_{i-n+1}\rvert
    \end{align} be the polar decompositions. As before, since $ S_{0}\cdots S_{-n+1} $ and $ T_{i}\cdots T_{i-n+1} $ are quasi-invertible, partial isometries $ V_{-n} $ and $ W_{-n,i} $ ($ n\in \mathbb{N} $, $ i\in \mathbb{Z} $) are unitary operators. We infer from \cite[Exercise 7.26(c)]{Wei1} that for every $ n\in \mathbb{N}_{1} $ and $ i\in \mathbb{Z} $,
    \begin{align}
        \label{ProofFormPolarDecompAdjointS}
        (S_{0}\cdots S_{-n+1})^{\ast} &= S_{-n+1}^{\ast}\cdots S_{0}^{\ast} = V_{-n}^{\ast}\lvert S_{-n+1}^{\ast}\cdots S_{0}^{\ast}\rvert\\
        \label{ProofFormPolarDecompAdjointT}
        (T_{i}\cdots T_{i-n+1})^{\ast} &= T_{i-n+1}^{\ast}\cdots T_{i}^{\ast} = W_{-n,i}^{\ast} \lvert T_{i-n+1}^{\ast}\cdots T_{i}^{\ast}\rvert
    \end{align}
    are the polar decompositions of $ (S_{0}\cdots S_{-n+1})^{\ast} $ and $ (T_{i}\cdots T_{i-n+1})^{\ast} $, respectively. We have
    \begin{align*}
        &W_{-n,i}^{\ast}\lvert T_{i-n+1}^{\ast}\cdots T_{i}^{\ast}\rvert U_{i,0} \stackrel{\eqref{ProofFormPolarDecompAdjointT}}{=} T_{i-n+1}^{\ast}\cdots T_{i}^{\ast} U_{i,0}\\
        &\stackrel{\eqref{ProofFormShiftInterweavingNegative}}{=} U_{i-n,-n}S_{-n+1}^{\ast}\cdots S_{0}^{\ast} \stackrel{\eqref{ProofFormPolarDecompAdjointS}}{=} U_{i-n,-n}V_{-n}^{\ast}\lvert S_{-n+1}^{\ast}\cdots S_{0}^{\ast}\rvert, \qquad n\in \mathbb{N}_{1}, \ i\in \mathbb{Z}.
    \end{align*}
    Thus, for every $ n\in \mathbb{N}_{1}, \ i\in \mathbb{Z} $,
    \begin{equation*}
        W_{-n,i}U_{i-n,-n}V_{-n}^{\ast}\lvert S_{-n+1}^{\ast}\cdots S_{0}^{\ast}\rvert = \lvert T_{i-n+1}^{\ast}\cdots T_{i}^{\ast}\rvert U_{i,0} \stackrel{\eqref{ProofFormModulusOfAdjoint}}{=} U_{i,0}\lvert S_{0-n+1}^{\ast}\cdots S_{0}^{\ast}\rvert.
    \end{equation*}
    Since $ \overline{\mathcal{R}(\lvert S_{-n+1}^{\ast}\cdots S_{0}^{\ast}\rvert)} = H $, by \eqref{ProofFormModulusOfAdjoint} we have $ W_{-n, i}U_{i-n,-n}V_{-n}^{\ast} = U_{i,0} $. Equivalently,
    \begin{equation*}
        U_{i-n,-n} = W_{-n,i}^{\ast}U_{i,0}V_{-n}, \qquad n\in \mathbb{N}_{1}, \ i\in \mathbb{Z}.
    \end{equation*}
    Hence,
    \begin{equation}
        \label{ProofFormEntriesOfUnitaryNegativeColumnsByZeroColumn}
        U_{-n} = F^{-n}D[(W_{-n,i}^{\ast})_{i\in \mathbb{Z}}]U_{0}V_{-n}, \qquad n\in \mathbb{N}_{1}.
    \end{equation}
    Now we will prove $ \eqref{FormWanderingProperty} $. First, observe that if $ k\in \mathbb{N}_{1} $, then
    \begin{align}
        \notag &U_{k}S_{k}\cdots S_{1} \stackrel{\eqref{ProofFormEntriesOfUnitaryPositiveColumnsByZeroColumn}}{=} F^{k}D[(W_{k,i})_{i\in \mathbb{Z}}] U_{0}V_{k}^{\ast} S_{k}\cdots S_{1} \\
        \notag&= F^{k}D[(W_{k,i})_{i\in \mathbb{Z}}] U_{0}\lvert S_{k}\cdots S_{1}\rvert \\
        \notag&\stackrel{\eqref{ProofFormModulusOfShift}}{=} F^{k}D[(W_{k,i})_{i\in \mathbb{Z}}]D[(\lvert T_{i+k}\cdots T_{i+1}\rvert)_{i\in \mathbb{Z}}]U_{0}\\
        \label{ProofFormEntriesOfUnitaryForDenseSubsetPositive}
        \notag&\stackrel{\eqref{ProofFormPolarDecompositionPositiveT}}{=} F^{k}D[(T_{i+k}\cdots T_{i+1})_{i\in \mathbb{Z}}]U_{0} \\
        &\stackrel{\text{Lemma }\ref{LemBasicPropertiesOfOperatorShifts}}{=} T^{k}U_{0}.
    \end{align}
    In turn, if $ -k\in \mathbb{N}_{1} $, then
    \begin{align}
        \notag &U_{k}S_{k+1}^{\ast}\cdots S_{0}^{\ast} \stackrel{\eqref{ProofFormEntriesOfUnitaryNegativeColumnsByZeroColumn}}{=} F^{k}D[(W_{k,i}^{\ast})_{i\in \mathbb{Z}}]U_{0}V_{k}S_{k+1}^{\ast}\cdots S_{0}^{\ast}\\
        \notag&= F^{k}D[(W_{k,i}^{\ast})_{i\in \mathbb{Z}}]U_{0}\lvert S_{k+1}^{\ast}\cdots S_{0}^{\ast}\rvert\\
        \notag &\stackrel{\eqref{ProofFormModulusOfAdjoint}}{=} F^{k}D[(W_{k,i}^{\ast})_{i\in \mathbb{Z}}]D[(\lvert T_{i+k+1}^{\ast}\cdots T_{i}^{\ast}\rvert)_{i\in \mathbb{Z}}]U_{0}\\
        \label{ProofFormEntriesOfUnitaryForDenseSubsetNegative}
        \notag&\stackrel{\eqref{ProofFormPolarDecompAdjointT}}{=} F^{k}D[(T_{i+k+1}^{\ast}\cdots T_{i}^{\ast})_{i\in \mathbb{Z}}]U_{0} \\
        &\stackrel{\text{Lemma }\ref{LemBasicPropertiesOfOperatorShifts}}{=} T^{\ast -k} U_{0}.
    \end{align}
    Since $ Ux^{(k)} = U_{k}x $ for all $ k\in \mathbb{Z} $ and $ x\in H $, we obtain that $ \mathcal{R}(U_{k})\perp \mathcal{R}(U_{n}) $ for all $ k,n\in \mathbb{Z} $, $ n\not= k $. In particular, for all $ k,m\in \mathbb{Z} $, $ k\not=m $, $ \langle U_{k}\widehat{x_{k}}, U_{m}\widehat{x_{m}}\rangle = 0 $, where
    \begin{equation*}
        \widehat{x_{j}} = \begin{cases}
            S_{j}\cdots S_{1} x_{j}, & \text{if } j > 0\\
            x_{0}, & \text{if } j = 0\\
            S_{j+1}^{\ast}\cdots S_{0}^{\ast}x_{j}, & \text{if } j <0 
        \end{cases} \qquad j=k,m,
    \end{equation*}
    and $ x_{k},x_{m}\in H $. But from the above computation it follows that $ U_{j}\widehat{x_{j}} = T^{[j]}U_{0}x_{j} $, $ j=k,m $. Hence, \eqref{FormWanderingProperty} holds.\par Finally, since $ U(x_{i})_{i\in \mathbb{Z}} = \sum_{i\in \mathbb{Z}} U_{i}x_{i} $, we have
    \begin{equation}
        \label{ProofFormLinearSpanOfEntries}
        \ell^{2}(\mathbb{Z},H) = \bigvee_{i\in \mathbb{Z}} \mathcal{R}(U_{i}).
    \end{equation}
    By the fact that $ \overline{\mathcal{R}(S_{n}\cdots S_{1})} = H $ and $ \overline{\mathcal{R}(S_{-n+1}^{\ast}\cdots S_{0}^{\ast})} = H $ for $ n\in \mathbb{N}_{1} $, we obtain from \eqref{ProofFormEntriesOfUnitaryPositiveColumnsByZeroColumn} and \eqref{ProofFormEntriesOfUnitaryNegativeColumnsByZeroColumn} that $ \mathcal{R}(U_{n}) = \overline{\mathcal{R}(U_{n}S_{n}\cdots S_{1})} $ and $ \mathcal{R}(U_{-n}) =\linebreak \overline{\mathcal{R}(U_{-n}S_{-n+1}^{\ast}\cdots S_{0}^{\ast})} $ for all $ n\in \mathbb{N}_{1} $. Combining the above and the equalities \eqref{ProofFormEntriesOfUnitaryForDenseSubsetPositive} and \eqref{ProofFormEntriesOfUnitaryForDenseSubsetNegative} we obtain that $ \mathcal{R}(U_{n}) = \overline{\mathcal{R}(T^{[n]}U_{0})} $ for all $ n\in \mathbb{Z} $. Hence, \eqref{ProofFormLinearSpanOfEntries} takes the form
    \begin{equation*}
        \ell^{2}(\mathbb{Z},H) = \bigvee_{i\in \mathbb{Z}} \overline{\mathcal{R}(T^{[i]}U_{0})} = \bigvee_{i\in \mathbb{Z}} \mathcal{R}(T^{[i]}U_{0}).
    \end{equation*}
    Therefore, \eqref{FormClosedSpan} holds.\\
    (ii). For $ n\in \mathbb{N}_{1} $ let $ U_{n} $ be defined by \eqref{ProofFormEntriesOfUnitaryPositiveColumnsByZeroColumn} and let $ U_{-n} $ be defined by \eqref{ProofFormEntriesOfUnitaryNegativeColumnsByZeroColumn}, where $ U_{0} $ is as in (ii). Since $ U_{0} $ is an isometry, so is $ U_{n} $ for every $ n\in \mathbb{Z} $. By the fact that $ \overline{\mathcal{R}(S_{n}\cdots S_{1})} = H $ and $ \overline{\mathcal{R}(S_{-n+1}^{\ast}\cdots S_{0}^{\ast})} = H $, using \eqref{ProofFormEntriesOfUnitaryForDenseSubsetPositive} and \eqref{ProofFormEntriesOfUnitaryForDenseSubsetNegative}, we obtain from \eqref{FormWanderingProperty} and \eqref{FormClosedSpan} that $ \mathcal{R}(U_{k})\perp \mathcal{R}(U_{n}) $ for all $ n,k\in \mathbb{Z} $, $ n\not=k $, and that
    \begin{equation}
        \label{ProofFormClosedSpan}
        \ell^{2}(\mathbb{Z},H) =\bigvee_{i\in \mathbb{Z}}\mathcal{R}(U_{i}).
    \end{equation}
    Define
    \begin{equation*}
        U\!: \ell^{2}(\mathbb{Z},H) \ni (x_{i})_{i\in \mathbb{Z}}\longmapsto \sum_{i\in \mathbb{Z}}U_{i}x_{i} \in \ell^{2}(\mathbb{Z},H). 
    \end{equation*}
    Since all $ U_{n} $'s ($ n\in \mathbb{Z} $) are isometries with mutually orthogonal  ranges, $ U $ is also isometric. By \eqref{ProofFormClosedSpan}, $ U $ is unitary. It remains to show that $ US = TU $. It is enough to prove that
    \begin{equation}
        \label{ProofFormShiftInterweavingOnEveryCoordinate}
        USx^{(k)} = TUx^{(k)}, \qquad x\in H, \ k\in \mathbb{Z}.
    \end{equation}
    First, we check \eqref{ProofFormShiftInterweavingOnEveryCoordinate} for $ k = 0 $. In this case we have
    \begin{align*}
        USx^{(0)} &= U(S_{1}x)^{(1)} = U_{1}S_{1}x \stackrel{\eqref{ProofFormEntriesOfUnitaryPositiveColumnsByZeroColumn}}{=} FD[(W_{1,i})_{i\in \mathbb{Z}}]U_{0}V_{1}^{\ast}S_{1}x\\
        &\stackrel{\eqref{ProofFormPolarDecompositionPositiveS}}{=}FD[(W_{1,i})_{i\in \mathbb{Z}}]U_{0}\lvert S_{1}\rvert x \\
        &\stackrel{\eqref{FormPositiveWeights}}{=} F[(W_{1,i})_{i\in \mathbb{Z}}]D[(\lvert T_{i+1}\rvert)_{i\in \mathbb{Z}}]U_{0}x \\
        &\stackrel{\eqref{ProofFormPolarDecompositionPositiveT}}{=} FD[(T_{i+1})_{i\in \mathbb{Z}}]U_{0}x = TU_{0}x = TUx^{(0)}.
    \end{align*}
    Next, assume $ k\in \mathbb{N}_{1} $. Since $ \overline{\mathcal{R}(S_{k}\cdots S_{1})} = H $, it is enough to check \eqref{ProofFormShiftInterweavingOnEveryCoordinate} for vectors of the form $ S_{k}\cdots S_{1}x $, where $ x\in H $. We have
    \begin{align*}
        &US(S_{k}\cdots S_{1}x)^{(k)} = U(S_{k+1}\cdots S_{1}x)^{(k+1)} = U_{k+1}S_{k+1}\cdots S_{1}x\\
        &\stackrel{\eqref{ProofFormEntriesOfUnitaryForDenseSubsetPositive}}{=} T^{k+1}U_{0}x = TT^{k}U_{0}x \stackrel{\eqref{ProofFormEntriesOfUnitaryForDenseSubsetPositive}}{=} TU_{k}S_{k}\cdots S_{1}x\\
        &= TU(S_{k}\cdots S_{1}x)^{(k)}.
    \end{align*}
    Now suppose that $ -k\in \mathbb{N}_{1} $. Note that it is enough to verify that
    \begin{equation}
        \label{ProofFormCoordinateInterwavingAdjoints}
        S_{k+1}^{\ast}U_{k+1}^{\ast} = U_{k}^{\ast}T^{\ast}.
    \end{equation}
    Indeed, if \eqref{ProofFormCoordinateInterwavingAdjoints} holds, then by taking adjoints we obtain $ U_{k+1}S_{k+1} = TU_{k} $. But $ U_{k+1}S_{k+1}x = U(S_{k+1}x)^{(k+1)} = USx^{(k)} $ and $ TUx^{(k)} = TU_{k}x $ for all $ x\in H $. Again, since $ \overline{\mathcal{R}(T_{i+1}^{\ast}\cdots T_{i-k-1}^{\ast})} = H $ for all $ i\in \mathbb{Z} $, it suffices to verify \eqref{ProofFormCoordinateInterwavingAdjoints} for vectors of the form $ (T_{i+1}^{\ast}\cdots T_{i-k-1}^{\ast}x_{i-k-1})_{i\in \mathbb{Z}} $, where $ (x_{i})_{i\in \mathbb{Z}}\in \ell^{2}(\mathbb{Z},H) $. For such a vector we have
    \begin{align*}
        &U_{k}^{\ast}T^{\ast}(T_{i+1}^{\ast}\cdots T_{i-k-1}^{\ast}x_{i-k-1})_{i\in \mathbb{Z}} \stackrel{\text{Lemma }\ref{LemBasicPropertiesOfOperatorShifts}}{=} U_{k}^{\ast}(T_{i+1}^{\ast}\cdots T_{i-k}^{\ast}x_{i-k})_{i\in \mathbb{Z}}\\
        &\stackrel{\eqref{ProofFormEntriesOfUnitaryNegativeColumnsByZeroColumn}}{=} V_{k}^{\ast}U_{0}^{\ast}D[(W_{k,i})_{i\in \mathbb{Z}}]F^{-k}(T_{i+1}^{\ast}\cdots T_{i-k}^{\ast}x_{i-k})_{i\in \mathbb{Z}} \\
        &= V_{k}^{\ast}U_{0}^{\ast}D[(W_{k,i})_{i\in \mathbb{Z}}](T_{i+k+1}^{\ast}\cdots T_{i}^{\ast}x_{i})_{i\in \mathbb{Z}} \\
        &\stackrel{\eqref{ProofFormPolarDecompAdjointT}}{=} V_{k}^{\ast}U_{0}^{\ast}(\lvert T_{i+k+1}^{\ast}\cdots T_{i}^{\ast}\rvert x_{i} )_{i\in \mathbb{Z}}\\
        &\stackrel{\eqref{FormNegativeWeights}}{=} V_{k}^{\ast}\lvert S_{k+1}^{\ast}\cdots S_{0}^{\ast}\rvert U_{0}^{\ast}(x_{i})_{i\in \mathbb{Z}} \\
        &= (S_{k+1}^{\ast}\cdots S_{0}^{\ast})U_{0}^{\ast}(x_{i})_{i\in \mathbb{Z}}
    \end{align*}
    and
    \begin{align*}
        &S_{k+1}^{\ast}U_{k+1}^{\ast}(T_{i+1}^{\ast}\cdots T_{i-k-1}^{\ast}x_{i-k-1})_{i\in \mathbb{Z}} \\
        &\stackrel{\eqref{ProofFormEntriesOfUnitaryNegativeColumnsByZeroColumn}}{=} S_{k+1}^{\ast}V_{k+1}^{\ast}U_{0}^{\ast}D[(W_{k+1,i})_{i\in \mathbb{Z}}]F^{-(k+1)}(T_{i+1}^{\ast}\cdots T_{i-k-1}^{\ast}x_{i-k-1})_{i\in \mathbb{Z}}\\
        &=S_{k+1}^{\ast}V_{k+1}^{\ast}U_{0}^{\ast}D[(W_{k+1,i})_{i\in \mathbb{Z}}](T_{i+k+2}^{\ast}\cdots T_{i}^{\ast}x_{i})_{i\in \mathbb{Z}} \\
        &\stackrel{\eqref{ProofFormPolarDecompAdjointT}}{=} S_{k+1}^{\ast}V_{k+1}^{\ast}U_{0}^{\ast}(\lvert T_{i+k+2}^{\ast}\cdots T_{i}^{\ast}\rvert x_{i})_{i\in \mathbb{Z}}\\
        &\stackrel{\eqref{FormNegativeWeights}}{=}S_{k+1}^{\ast}V_{k+1}^{\ast}\lvert S_{k+2}^{\ast}\cdots S_{0}^{\ast}\rvert U_{0}^{\ast}(x_{i})_{i\in \mathbb{Z}} \\
        &= S_{k+1}^{\ast}S_{k+2}^{\ast}\cdots S_{0}^{\ast}U_{0}^{\ast}(x_{i})_{i\in \mathbb{Z}}.
    \end{align*}
    Hence, \eqref{ProofFormCoordinateInterwavingAdjoints} holds, what completes the proof.
\end{proof}
\begin{remark}
    \label{RemFormulaForUnitaryOperator}
    The careful look on the proof of Theorem \ref{ThmGeneralCharacterizationOfUnitaryEquivalence} reveals that all entries in the matrix representation of the unitary operator $ U $ making $ S $ and $ T $ unitarily equivalent are uniquely determined by $ U_{0} $ and that the isometry $ U_{0} $ in (ii) is exactly the zeroth column of the unitary operator we construct. The reader can easily verify that the choice of the zeroth column is arbitrary; we may prove that all entries of $ U $ are uniquely determined by any other fixed column as well.
\end{remark}
As a corollary we obtain the counterpart of \cite[Corollary 3.2]{Lam3} for bilateral shifts.
\begin{corollary}
    \label{CorUnitaryEquivalenceShiftsWithUnitaryWeights}
    Let $ H $ be a Hilbert space. Let $ (S_{i})_{i\in \mathbb{Z}}, (T_{i})_{i\in \mathbb{Z}} \subset \mathbf{B}(H) $ be two sequences of unitary operators on $ H $ and let $ S,T $ be the bilateral weighted shifts with weights $ (S_{i})_{i\in \mathbb{Z}} $ and $ (T_{i})_{i\in \mathbb{Z}} $, respectively. Then $ S $ and $ T $ are unitarily equivalent.
\end{corollary}
\begin{proof}
    Define $ U_{0}\!: H\to \ell^{2}(\mathbb{Z},H) $ by
    \begin{equation*}
        U_{0}x = x^{(0)}, \qquad x\in H.
    \end{equation*}
    It is a matter of routine to verify that \eqref{FormPositiveWeights}-\eqref{FormClosedSpan} hold, so we can apply Theorem \ref{ThmGeneralCharacterizationOfUnitaryEquivalence}.(ii).
\end{proof}
It turns out that under some additional assumptions on weights, we can simplify the conditions in Theorem \ref{ThmGeneralCharacterizationOfUnitaryEquivalence}.
\begin{corollary}
    \label{CorUnitaryEquivalenceShiftsWithDoublyCommutingWeights}
    Suppose that $ S $ and $ T $ are as in Theorem \ref{ThmGeneralCharacterizationOfUnitaryEquivalence} with additional assumption that
    the sequences $ (S_{i})_{i\in \mathbb{Z}} $ and $ (T_{i})_{i\in \mathbb{Z}} $ are doubly commuting. 
    \begin{enumerate}
        \item if $ US = TU $ for a unitary operator $ U\in \mathbf{B}(\ell^{2}(\mathbb{Z},H)) $, then the isometry $ U_{0}\in \mathbf{B}(H,\ell^{2}(\mathbb{Z},H)) $ as in Theorem \ref{ThmGeneralCharacterizationOfUnitaryEquivalence}.(i) satisfies
        \begin{align}
            \label{FormPositiveWeightsDoublyCommuting}
            U_{0}\lvert S_{n}\rvert = D[(\lvert T_{i+n}\rvert)_{i\in \mathbb{Z}}]U_{0}, \qquad n\in \mathbb{N}_{1},\\
            \label{FormNegativeWeightsDoublyCommuting}
            U_{0}\lvert S_{-n+1}^{\ast}\rvert = D[(\lvert T_{i-n+1}^{\ast}\rvert)_{i\in \mathbb{Z}}]U_{0}, \qquad n\in \mathbb{N}_{1}.
        \end{align}
        \item If there exists an isometry $ U_{0}\in \mathbf{B}(H,\ell^{2}(\mathbb{Z},H)) $ satisfying \eqref{FormPositiveWeightsDoublyCommuting}, \eqref{FormNegativeWeightsDoublyCommuting}, \eqref{FormWanderingProperty} and \eqref{FormClosedSpan}, then $ S $ and $ T $ are unitarily equivalent.
    \end{enumerate}
\end{corollary}
\begin{proof}
    Since $ (S_{i})_{i\in \mathbb{Z}} $ and $ (T_{i})_{i\in\mathbb{Z}} $ are sequences of doubly commuting operators, by \cite[Lemma 2.1]{AnaChaJabSto1} we have
    \begin{align}
        \label{ProofFormModulusForDoublyCommutingPositiveWeights}
        &\lvert S_{n}\rvert\cdots \lvert S_{1}\rvert = \lvert S_{n}\cdots S_{1}\rvert, \qquad n\in \mathbb{N}_{1},\\
        \label{ProofFormModulusForDoublyCommutingPositiveWeights2}
        &\lvert T_{i+n}\cdots T_{i+1}\rvert = \lvert T_{i+n}\rvert\cdots \lvert T_{i+1}\rvert, \qquad n\in \mathbb{N}_{1}, \ i\in \mathbb{Z}
    \end{align}
    (i). By Theorem \ref{ThmGeneralCharacterizationOfUnitaryEquivalence}, the isometry $ U_{0}\!: H\to \ell^{2}(\mathbb{Z},H) $ satisfies \eqref{FormPositiveWeights}-\eqref{FormClosedSpan}. We check that \eqref{FormPositiveWeightsDoublyCommuting} and \eqref{FormNegativeWeightsDoublyCommuting} hold. By \eqref{ProofFormModulusForDoublyCommutingPositiveWeights} and \eqref{ProofFormModulusForDoublyCommutingPositiveWeights2}, \eqref{FormPositiveWeights} takes the form
    \begin{equation}
        \label{ProofFormPositiveWeightsForDoublyCommuting}
        U_{0}\lvert S_{n}\rvert\cdots \lvert S_{1}\rvert = D[(\lvert T_{i+n}\rvert)_{i\in \mathbb{Z}}]\cdots D[(\lvert T_{i+1}\rvert)]U_{0}, \qquad n\in \mathbb{N}_{1}.
    \end{equation}
    If $ n = 1 $, then \eqref{ProofFormPositiveWeightsForDoublyCommuting} coincides with \eqref{FormPositiveWeightsDoublyCommuting}. If $ n\in \mathbb{N}_{2} $, then, using \eqref{ProofFormPositiveWeightsForDoublyCommuting} with $ n $ replaced with $ n-1 $, we obtain
    \begin{equation*}
        U_{0}\lvert S_{n}\rvert\lvert S_{n-1}\rvert\cdots \lvert S_{1}\rvert = D[(\lvert T_{i+n}\rvert)_{i\in \mathbb{Z}}]U_{0}\lvert S_{n-1}\rvert\cdots \lvert S_{1}\rvert.
    \end{equation*}
    Since $ \overline{\mathcal{R}(\lvert S_{n-1}\rvert\cdots \lvert S_{1}\rvert)} = H $, the above equality implies \eqref{FormPositiveWeightsDoublyCommuting}. Similar reasoning shows that \eqref{FormNegativeWeights} implies \eqref{FormNegativeWeightsDoublyCommuting}. \\
    (ii). Arguing as in (i) we can prove that \eqref{FormPositiveWeights} and \eqref{FormNegativeWeights} can be derived from \eqref{FormPositiveWeightsDoublyCommuting} and \eqref{FormNegativeWeightsDoublyCommuting}.
\end{proof}
Corollary \ref{CorUnitaryEquivalenceShiftsWithDoublyCommutingWeights} gives (in particular) the characterization of unitary equivalence for classical bilateral weighted shifts (with scalar weights). However, in this case the proof is much simpler than the proof of Theorem \ref{ThmGeneralCharacterizationOfUnitaryEquivalence}; in particular, the unitary operator making two bilateral weighted shifts unitarily equivalent is always of diagonal form (see \cite[Theorem 1, p.53]{Shi1}).\par
Now, let us show how to deduce \cite[Corollary 2.4]{Kos1} from Theorem \ref{ThmGeneralCharacterizationOfUnitaryEquivalence}. 
\begin{corollary}
    \label{CorCharacterizationOfUnitaryEquivalenceByDiagonalForm}
    Suppose that $ S $ and $ T $ are as in Theorem \ref{ThmGeneralCharacterizationOfUnitaryEquivalence}. The following conditions are equivalent:
    \begin{enumerate}
        \item there exists a unitary operator $ U\in \mathbf{B}\left( \ell^{2}(\mathbb{Z},H) \right) $ of diagonal form such that $ US = TU $,
        \item there exist $ p\in \mathbb{Z} $ and a unitary operator $ U_{p,0}\in \mathbf{B}(H) $ such that
        \begin{align}
            \label{FormPositiveWeightsDiagonalForm}
            &\lVert S_{n}\cdots S_{1}v\rVert = \lVert T_{p+n}\cdots T_{p+1}U_{p,0}v\rVert, \qquad v\in H, \ n\in \mathbb{N}_{1}\\
            \label{FormNegativeWeightsDiagonalForm}
            &\lVert S_{-n+1}^{\ast}\cdots S_{0}^{\ast}v\rVert = \lVert T_{p-n+1}^{\ast}\cdots T_{p}^{\ast}U_{p,0}v\rVert, \qquad v\in H, \ n\in \mathbb{N}_{1}.
        \end{align}
    \end{enumerate}
\end{corollary}
\begin{proof}
    For $ p\in \mathbb{Z} $ and $ n\in \mathbb{N}_{1} $ let $ V_{n}, V_{-n} $, $ W_{n,p} $ and $ W_{-n,p} $ be as in \eqref{ProofFormPolarDecompositionPositiveS}, \eqref{ProofFormPolarDecompositionNegativeS}, \eqref{ProofFormPolarDecompositionPositiveT} and \eqref{ProofFormPolarDecompositionNegativeT}, respectively.
Since $ S_{n} $ and $ T_{n} $ are quasi-invertible, we infer from \eqref{FormQuasiInvertibleClosedUnderProducts} and \eqref{FormPolarDecompositionGeneral} that $ V_{n} $ and $ W_{n,p} $ are unitary for every $ n\in \mathbb{Z} $. \\
    (i)$ \Longrightarrow $(ii). By Theorem \ref{ThmGeneralCharacterizationOfUnitaryEquivalence}.(i) the isometry $ U_{0}\!: H\to \ell^{2}(\mathbb{Z},H) $ satisfies \eqref{FormPositiveWeights} and \eqref{FormNegativeWeights}. Since $ U $ is of diagonal form, $ U_{0} $ is of the form $ U_{0}x = (U_{p,0}x)^{(p)} $, where $ U_{p,0}\in \mathbf{B}(H) $ is a unitary operator and $ p\in \mathbb{Z} $. Then \eqref{FormPositiveWeights} takes the form
    \begin{align*}
        U_{p,0}V_{n}^{\ast} S_{n}\cdots S_{1} = W_{n,p}^{\ast} T_{p+n}\cdots T_{p+1} U_{p,0}, \qquad n\in \mathbb{N}_{1}.
    \end{align*}
    Thus, \eqref{FormPositiveWeightsDiagonalForm} easily follows from the above equality. In turn, using \eqref{ProofFormPolarDecompAdjointS} and \eqref{ProofFormPolarDecompAdjointT}, \eqref{FormNegativeWeights} takes the form
    \begin{align*}
        U_{p,0}V_{-n} S_{-n+1}^{\ast}\cdots S_{0}^{\ast} = W_{-n,p}T_{p-n+1}^{\ast}\cdots T_{p}^{\ast}U_{p,0}, \qquad n\in \mathbb{N}_{1}.
    \end{align*}
    From the above, we can easily deduce \eqref{FormNegativeWeightsDiagonalForm}.\\
    (ii)$ \Longrightarrow $(i).  Observe that for every $ x\in H $,
    \begin{align*}
        &\lVert \lvert S_{n}\cdots S_{1}\rvert x\rVert^{2} = \lVert V_{n}\lvert S_{n}\cdots S_{1}\rvert x\rVert^{2} = \lVert S_{n}\cdots S_{1} x\rVert^{2} \\
        &\stackrel{\eqref{FormPositiveWeightsDiagonalForm}}{=} \lVert T_{p+n}\cdots T_{p+1}U_{p,0}x\rVert^{2} = \lVert W_{n,p}\lvert T_{p+n}\cdots T_{p+1}\rvert U_{p,0}x\rVert^{2} \\
        &= \lVert \lvert T_{p+n}\cdots T_{p+1}\rvert U_{p,0}x\rVert^{2} = \lVert U_{p,0}^{\ast}\lvert T_{p+n}\cdots T_{p+1}\rvert U_{p,0}x\rVert^{2}.
    \end{align*}
    Hence,
    \begin{equation*}
        \langle \lvert S_{n}\cdots S_{1}\rvert^{2}x,x\rangle = \langle (U_{p,0}^{\ast}\lvert T_{p+n}\cdots T_{p+1}\rvert U_{p,0})^{2}x,x\rangle, \qquad x\in H,
    \end{equation*}
    which implies that $ \lvert S_{n}\cdots S_{1}\rvert^{2} = (U_{p,0}^{\ast}\lvert T_{p+n}\cdots T_{p+1}\rvert U_{p,0})^{2} $ for every $ n\in \mathbb{N}_{1} $. By the uniqueness of square root, $ \lvert S_{n}\cdots S_{1}\rvert = U_{p,0}^{\ast}\lvert T_{p+n}\cdots T_{p+1}\rvert U_{p,0} $. From this it follows that
    \begin{equation}
        \label{ProofFormPositiveModuliCommutesWithUnitary}
        U_{p,0} \lvert S_{n}\cdots S_{1}\rvert = \lvert T_{p+n}\cdots T_{p+1}\rvert U_{p,0}, \qquad n\in \mathbb{N}_{1}
    \end{equation}
    In a similar manner we prove that
    \begin{equation}
        \label{ProofFormNegativeModuliCommutesWithUnitary}
        U_{p,0}\lvert S_{-n+1}^{\ast}\cdots S_{0}^{\ast}\rvert = \lvert T_{p-n+1}^{\ast}\cdots T_{p}^{\ast}\rvert U_{p,0}, \qquad n\in \mathbb{N}.
    \end{equation} 
    Define $ U_{0}\!: H\to \ell^{2}(\mathbb{Z},H) $ by the formula: $ U_{0}x = (U_{p,0}x)^{(p)} $, $ x\in H $. Clearly, $ U_{0} $ is isometric. In view of \eqref{ProofFormPositiveModuliCommutesWithUnitary} and \eqref{ProofFormNegativeModuliCommutesWithUnitary}, \eqref{FormPositiveWeights} and \eqref{FormNegativeWeights} hold; \eqref{FormWanderingProperty} and \eqref{FormClosedSpan} are trivially satisfied. The application of Theorem \ref{ThmGeneralCharacterizationOfUnitaryEquivalence}.(ii) completes the proof.
\end{proof}
Next, we show a convenient necessary condition for bilateral weighted shifts to be unitarily equivalent by a unitary operator of diagonal form.
\begin{lemma}
    \label{LemNecessaryConditionEquivalenceByUnitaryOfDiagonalForm}
    Suppose that $ S $ and $ T $ are as in Theorem \ref{ThmGeneralCharacterizationOfUnitaryEquivalence}. Assume that there exists a unitary operator $ U = [U_{i,j}]_{i,j\in \mathbb{Z}}\in \mathbf{B}(\ell^{2}(\mathbb{Z},H)) $ of diagonal form such that $ US = TU $. Then there exists $ p\in \mathbb{Z} $ such that $ \lvert S_{i}\rvert $ and $ \lvert T_{i+p}\rvert $ are unitarily equivalent for all $ i\in \mathbb{Z} $; in particular, $ \sigma(\lvert S_{i}\rvert) = \sigma(\lvert T_{i+p}\rvert) $ for every $ i\in \mathbb{Z} $.
\end{lemma}
\begin{proof}
   Since $ U $ is of diagonal form, there exists $ p\in \mathbb{Z} $ such that $ U_{i+p,i} $ ($ i\in \mathbb{Z} $) are the only non-zero entries of the matrix representation of $ U $. This implies that $ U_{i+p,i} $ has to be unitary for every $ i\in \mathbb{Z} $. The equalities $ US = TU $ and $ US^{\ast} = T^{\ast}U $ implies that $ U\lvert S\rvert = \lvert T\rvert U $. Now, it follows from Lemma \ref{LemBasicPropertiesOfOperatorShifts} that
    \begin{equation*}
        U_{i+p,i}\lvert S_{i+1}\rvert = \lvert T_{i+p+1}\rvert U_{i+p,i}, \qquad i\in \mathbb{Z}.
    \end{equation*}
    Hence, $ \lvert S_{i}\rvert $ and $ \lvert T_{i+p}\rvert $ are unitarily equivalent for all $ i\in \mathbb{Z} $.
\end{proof}
It can be easily deduced from the Schur decomposition (see \cite[Theorem 2.31]{HorJoh1}) that if $ A\in \mathbf{B}(\mathbb{C}^{m}) $, then $ \sigma(\lvert A\rvert) = \{\lvert \lambda\rvert\!: \lambda \in \sigma(A) \} $. Hence, in case $ H = \mathbb{C}^{m} $ ($ m\in \mathbb{N}_{1} $) the necessary condition $ \sigma(\lvert S_{i}\rvert) = \sigma(\lvert T_{i+p}\rvert) $ in Lemma \ref{LemNecessaryConditionEquivalenceByUnitaryOfDiagonalForm} actually says that if the two bilateral weighted shifts $ S $ and $ T $ are unitarily equivalent by a unitary operator of diagonal form, then the moduli of eigenvalues of $ S_{i} $ and $ T_{i+p} $ are equal for some $ p\in \mathbb{Z} $ (see \cite[Proposition 2.7]{Kos1} for the proof in the case $ H = \mathbb{C}^{2} $). However, in \cite[Example 2.8]{Kos1} the author showed that this condition does not guarantee that the shifts are unitarily equivalent by an operator of diagonal form, even if the weights are normal and commuting. The next result shed a new light on the aforementioned example.
\begin{corollary}
    Let $ H = \mathbb{C}^{m} $, $ m\in \mathbb{N}_{1} $. Let $ (S_{i})_{i\in \mathbb{Z}}, (T_{i})_{i\in \mathbb{Z}} \subset \mathbf{B}(H) $ be two uniformly bounded sequences of normal, invertible and commuting operators on $ H $ and let $ S,T $ be the bilateral weighted shifts with weights $ (S_{i})_{i\in \mathbb{Z}} $ and $ (T_{i})_{i\in \mathbb{Z}} $, respectively. The following conditions are equivalent:
    \begin{enumerate}
        \item $ S $ and $ T $ are unitarily equivalent by a unitary operator $ U\in \mathbf{B}(\ell^{2}(\mathbb{Z},H)) $ of diagonal form,
        \item there exist $ p\in \mathbb{Z} $ and two orthonormal basis $ (v_{n})_{n=1}^{m},(w_{n})_{n=1}^{m} $ of $ H $ such that
        \begin{enumerate}
            \item $ v_{n} $ is an eigenvector of $ \lvert S_{i}\rvert $ for all $ i\in \mathbb{Z} $ and $ n = 1,\ldots,m $,
            \item $ w_{n} $ is an eigenvector of $ \lvert T_{i}\rvert $ for all $ i\in \mathbb{Z} $ and $ n = 1,\ldots,m $,
            \item for every $ \lambda\in \mathbb{C} $ and every $ n = 1,\ldots,m $
            \begin{equation*}
                \lvert S_{i}\rvert v_{n} = \lambda v_{n}\iff \lvert T_{i+p}\rvert w_{n} = \lambda w_{n}, \qquad i\in \mathbb{Z}.
            \end{equation*}
        \end{enumerate}
    \end{enumerate}
\end{corollary}
\begin{proof}
    (i)$ \Longrightarrow $(ii). Since $ S_{i} $'s ($ i\in \mathbb{Z} $) are normal and commuting, we infer from Fuglede-Putnam theorem (see \cite[Theorem 6.7]{Con1}), that they are doubly commuting; the same holds for $ T_{i} $'s ($ i\in \mathbb{Z} $). It follows from Corollary \ref{CorUnitaryEquivalenceShiftsWithDoublyCommutingWeights}.(i) that the isometry $ U_{0}\!: H\to \ell^{2}(\mathbb{Z},H) $ satisfies \eqref{FormPositiveWeightsDoublyCommuting}, \eqref{FormNegativeWeightsDoublyCommuting}, \eqref{FormWanderingProperty} and \eqref{FormClosedSpan}. By (i), $ U_{0} $ is such that $ U_{i,0} = 0 $ for all $ i\in \mathbb{Z}\setminus \{p\} $ and $ U_{p,0} $ is unitary with some $ p\in \mathbb{Z} $. Since $ \lvert S_{i}\rvert $'s ($ i\in \mathbb{Z} $) are normal and commuting, we deduce from \cite[Theorem 2.5.5]{HorJoh1} that there exists an orthonormal basis $ (v_{n})_{n=1}^{m} $ such that
    \begin{equation*}
        \lvert S_{i}\rvert v_{n} = \lambda_{i,n}v_{n}, \qquad i\in \mathbb{Z}, \ n=1,\ldots,m;
    \end{equation*} in particular, (a) holds. Using \eqref{FormPositiveWeightsDoublyCommuting} and \eqref{FormNegativeWeightsDoublyCommuting}, we obtain
    \begin{equation*}
        \lambda_{i,n} U_{p,0}v_{n} = U_{p,0}\lvert S_{i}\rvert v_{n} = \lvert T_{p+i}\rvert U_{p,0}v_{n}, \qquad i\in \mathbb{Z}, \ n=1,\ldots,m. 
    \end{equation*}
    By the above, $ (U_{p,0}v_{n})_{n=1}^{m} $ is the orthonormal basis of $ H $ satisfying (b) and (c).\\
    (ii)$ \Longrightarrow $(i). Define $ U_{p,0}\in \mathbf{B}(H) $ as follows:
    \begin{equation*}
        U_{p,0}v_{n} = w_{n}, \qquad n=1,\ldots, m.
    \end{equation*}
    Clearly, $ U_{p,0} $ is unitary. For every $ n=1,\ldots,m $, if $ S_{i}v_{n} = \lambda_{i,n}v_{n} $, then
    \begin{align*}
        U_{p,0}\lvert S_{i}\rvert v_{n} = \lambda_{i,n}U_{p,0}v_{n} = \lambda_{i,n}w_{n} \stackrel{(c)}{=} \lvert T_{i+p}\rvert w_{n} = \lvert T_{i+p}\rvert U_{p,0}v_{n}.
    \end{align*}
    Define $ U_{0}\!: H\to \ell^{2}(\mathbb{Z},H) $ by the formula:
    \begin{equation*}
        U_{0}x = (U_{p,0}x)^{(p)}, \qquad x\in H.
    \end{equation*}
    Then $ U_{0} $ is an isometry satisfying \eqref{FormPositiveWeightsDoublyCommuting} and \eqref{FormNegativeWeightsDoublyCommuting}. The conditions \eqref{FormWanderingProperty} and \eqref{FormClosedSpan} trivially hold. Applying Corollary \ref{CorUnitaryEquivalenceShiftsWithDoublyCommutingWeights}.(ii) we get (i).
\end{proof}
Now let us present another characterization of unitary equivalence of bilateral shifts with operator weights in terms of factors in their polar decompositions.
\begin{theorem}
    \label{ThmCharacterizationOfUnitaryEquivalenceByPolarDecompositions}
    Suppose that $ S $ and $ T $ are as in Theorem \ref{ThmGeneralCharacterizationOfUnitaryEquivalence}. Let $ S = V_{S}\lvert S\rvert $ and $ T = V_{T}\lvert T\rvert $ be the polar decompositions of $ S $ and $ T $, respectively. For a unitary operator $ U\!: \ell^{2}(\mathbb{Z},H)\to \ell^{2}(\mathbb{Z},H) $ the following conditions are equivalent:
    \begin{enumerate}
        \item $ US = TU $,
        \item $ U\lvert S\rvert = \lvert T\rvert U $ and $ UV_{S} = V_{T}U $.
    \end{enumerate}
\end{theorem}
\begin{proof}
    (i)$ \Longrightarrow $(ii). Repeating an argument as in the proof of Theorem \ref{ThmGeneralCharacterizationOfUnitaryEquivalence}, we have $ U\lvert S\rvert = \lvert T\rvert U $. Moreover,
    \begin{align*}
        UV_{S}\lvert S\rvert = US = TU = V_{T}\lvert T\rvert U = V_{T}U\lvert S\rvert.
    \end{align*}
    Since $ \overline{\mathcal{R}(\lvert S\rvert)} = \overline{\mathcal{R}(D[(\lvert S_{i+1}\rvert)_{i\in \mathbb{Z}}])} = \ell^{2}(\mathbb{Z},H) $, it follows that $ UV_{S} = V_{T}U $.\\
    (ii)$ \Longrightarrow $(i). We have $ US = UV_{S}\lvert S\rvert = V_{T}U\lvert S\rvert = V_{T}\lvert T\rvert U = TU $.
\end{proof}
By Lemma \ref{LemPolarDecompositionOfShift}, the operators $ V_{S} $ and $ V_{T} $ in the above theorem are weighted shifts with unitary weights. By Corollary \ref{CorUnitaryEquivalenceShiftsWithUnitaryWeights}, the equality $ UV_{S} = V_{T}U $ always holds with some unitary operator $ U $; however, this operator $ U $ does not have to satisfy $ U\lvert S\rvert = \lvert T\rvert U $. Since the unitary operator $ U $ making $ S $ and $ T $ unitarily equivalent can be found among all unitary operators intertwining $ V_{s} $ and $ V_{T} $, the possible further research is to describe all unitary operators, which intertwine two shifts with unitary weights. \par
The rest of this section is devoted to study unitary equivalence of shifts with positive commuting weights. First, we prove that unitary operator intertwining bilateral shifts with positive weights has to be constant on diagonals (see \cite[Lemma 5]{Pili1} for similar result for self-adjoint operator intertwining bilateral shifts).
\begin{lemma}
    \label{LemNonnegativeWeightsUnitaryConstantOnDiagonals}
    Let $ H $ be a Hilbert space. Let $ (S_{i})_{i\in \mathbb{Z}}, (T_{i})_{i\in \mathbb{Z}} \subset \mathbf{B}(H) $ be two uniformly bounded sequences of positive operators on $ H $ and denote by $ S,T $ the bilateral weighted shifts with weights $ (S_{i})_{i\in \mathbb{Z}} $ and $ (T_{i})_{i\in \mathbb{Z}} $, respectively. Suppose $ U = [U_{i,j}]_{i,j\in \mathbb{Z}} \in \mathbf{B}(\ell^{2}(\mathbb{Z},H)) $ is a unitary operator such that $ US = TU $. Then $ U $ is constant on diagonals, that is, $ U_{i+1,j+1} = U_{i,j} $ for all $ i,j\in \mathbb{Z} $.
\end{lemma}
\begin{proof}
    From Theorem \ref{ThmCharacterizationOfUnitaryEquivalenceByPolarDecompositions} we obtain that $ U\lvert S\rvert = \lvert T\rvert U $. Since the weights of $ S $ and $ T $ are positive, by Lemma \ref{LemBasicPropertiesOfOperatorShifts}.(iii) we have $ U_{i,j}S_{j+1} = T_{i+1}U_{i,j} $ for all $ i,j\in \mathbb{Z} $. On the other hand, from Lemma \ref{LemInterweavingWeightedShifts} it follows that $ U_{i+1,j+1} S_{j+1} = T_{i+1}U_{i,j} $ for all $ i,j\in \mathbb{Z} $. Since $ \overline{\mathcal{R}(S_{j})} = H $ for every $ j\in \mathbb{Z} $, we obtain that $ U_{i+1,j+1} = U_{i,j} $ for all $ i,j\in \mathbb{Z} $.
\end{proof}
The next result states that if $ H $ is two dimensional, then under certain assumptions on weights the unitary equivalence is always given by the operator having at most two non-zero diagonal.
\begin{theorem}
    \label{Thm2DimUnitaryEquivalenceAtMostTwoNonZeroDiagonals}
    Suppose $ H = \mathbb{C}^{2} $. Let $ (S_{i})_{i\in \mathbb{Z}}, (T_{i})_{i\in \mathbb{Z}} \subset \mathbf{B}(H) $ be two uniformly bounded sequences of positive and commuting operators on $ H $ and assume additionally that every $ S_{j} $ and every $ T_{j} $ ($ j\in \mathbb{Z} $) has two distinct eigenvalues. Denote by $ S,T $ the bilateral weighted shifts with weights $ (S_{i})_{i\in \mathbb{Z}} $ and $ (T_{i})_{i\in \mathbb{Z}} $, respectively. The following conditions are equivalent:
    \begin{enumerate}
        \item $ S $ and $ T $ are unitarily equivalent,
        \item $ S $ and $ T $ are unitarily equivalent by a unitary operator with at most two non-zero diagonals.
    \end{enumerate}
\end{theorem}
Before we state the proof we need two lemmata.
\begin{lemma}
    \label{LemZeroColumnEigenvectorsMapping}
    Let $ H $, $ (S_{i})_{i\in \mathbb{Z}}, (T_{i})_{i\in \mathbb{Z}} $ be as in the assumptions of Theorem \ref{Thm2DimUnitaryEquivalenceAtMostTwoNonZeroDiagonals}. Suppose $ U = [U_{i,j}]_{i,j\in \mathbb{Z}} \in \mathbf{B}(\ell^{2}(\mathbb{Z},H)) $ is a unitary operator such that $ US = TU $. Let $ \{v_{1},v_{2}\}\subset H $ be the common orthonormal basis of eigenvectors for the sequence $ (S_{i})_{i\in \mathbb{Z}} $ and let $ \{w_{1},w_{2}\}\subset H $ be the common orthonormal basis of eigenvectors for the sequence $ (T_{i})_{i\in \mathbb{Z}} $. Then:
    \begin{enumerate}
        \item for every $ j\in \{1,2\} $ there exist $ k_{j}\in \{1,2\} $ and $ i_{j}\in \mathbb{Z} $ such that
        \begin{equation*}
            \langle U_{i_{j},0}v_{j},w_{k_{j}}\rangle \not= 0,
        \end{equation*}
        \item for every $ k\in \{1,2\} $ there exist $ j_{k}\in \{1,2\} $ and $ i_{k}\in \mathbb{Z} $ such that
        \begin{equation*}
            \langle U_{i_{k},0}v_{j},w_{k}\rangle \not= 0.
        \end{equation*}
    \end{enumerate}
\end{lemma}
\begin{proof}
    (i). For $ j\in \{1,2\} $ we have
    \begin{equation*}
        1 = \lVert Uv_{j}^{(0)}\rVert^{2} = \sum_{i\in \mathbb{Z}} \lVert U_{i,0}v_{j}\rVert^{2}.
    \end{equation*}
    This implies that there exists $ i_{j}\in \mathbb{Z} $ such that $ U_{i_{j},0}v_{j} \not=0 $. Since $ \{w_{1},w_{2}\} $ is the orthonormal basis of $ H $, there exists $ k_{j}\in \{1,2\} $ such that $ \langle U_{i_{j},0}v_{j},w_{k_{j}}\rangle \not=0 $. \\
    (ii). For $ k\in \{1,2\} $ we have
    \begin{align*}
        1 &= \lVert U^{\ast} w_{k}^{(0)}\rVert^{2} = \sum_{i\in \mathbb{Z}} \lVert (U^{\ast})_{i,0}w_{k}\rVert^{2} \\
        &= \sum_{i\in \mathbb{Z}} \lVert U_{0,i}^{\ast}w_{k}\rVert^{2} \stackrel{\text{Lemma }\ref{LemNonnegativeWeightsUnitaryConstantOnDiagonals}}{=} \sum_{i\in \mathbb{Z}} \lVert U_{-i,0}^{\ast}w_{k}\rVert^{2}.
    \end{align*}
    Arguing as in (i) we find $ i_{k}\in \mathbb{Z} $ and $ j_{k}\in \{1,2\} $ such that
    \begin{equation*}
        \langle v_{j_{k}}, U_{i_{k},0}^{\ast}w_{k}\rangle = \langle U_{i_{k},0}v_{j_{k}}, w_{k}\rangle \not=0,
    \end{equation*}
    which proves (ii).
\end{proof}
\begin{lemma}
    \label{Lem2DimBijection}
    Under the assumptions of Lemma \ref{LemZeroColumnEigenvectorsMapping} there exist a bijection $ \sigma\!: \{1,2\}\to \{1,2\} $ and a function $ \tau\!: \{1,2\}\to \mathbb{Z} $ such that
    \begin{equation*}
        \langle U_{\tau(j),0}v_{j}, w_{\sigma(j)}\rangle \not=0, \qquad j\in \{1,2\}.
    \end{equation*}
\end{lemma}
\begin{proof}
    For $ j\in \{1,2\} $ set
    \begin{equation*}
        B_{j} = \left\{ k\in \{1,2\}\!: \langle U_{i,0}v_{j},w_{k}\rangle \not= 0 \text{ for some } i\in \mathbb{Z} \right\}.
    \end{equation*}
    By Lemma \ref{LemZeroColumnEigenvectorsMapping}.(i), $ B_{j}\not=\varnothing $ for $ j\in \{1,2\} $. From Lemma \ref{LemZeroColumnEigenvectorsMapping}.(ii) we deduce that it is not possible that $ B_{1} = B_{2} = \{j\} $ for any $ j\in \{1,2\} $. The only possibilities for the sets $ B_{1}, B_{2} $ are listed in the following table:
    \begin{center}
        \begin{tabular}{c|c}
            $ B_{1} $ & $ B_{2} $\\
            \hline
            $ \{ 1\} $ & $ \{2 \} $ \\
            $ \{1 \} $ & $ \{1,2 \} $ \\   
            $ \{ 2\} $ & $ \{1 \} $ \\ 
            $ \{2 \} $ & $ \{ 1,2\} $ \\   
            $ \{1,2 \} $ & $ \{ 1\} $ \\   
            $ \{ 1,2\} $ & $ \{2 \} $ \\   
            $ \{ 1,2\} $ & $ \{ 1,2\} $ \\   
        \end{tabular}
    \end{center}
    It is a matter of routine to verify that in every case there exists a bijection $ \sigma\!: \{1,2\}\to \{1,2\} $ such that $ \sigma(j)\in B_{j} $ for $ j\in \{1,2\} $. The existence of the required function $ \tau\!: \{1,2\}\to \mathbb{Z} $ easily follows.
\end{proof}
Now we are in the position to prove Theorem \ref{Thm2DimUnitaryEquivalenceAtMostTwoNonZeroDiagonals}.
\begin{proof}[Proof of Theorem \ref{Thm2DimUnitaryEquivalenceAtMostTwoNonZeroDiagonals}]
    It is enough to prove that (i) implies (ii). Let $ \sigma\!: \{1,2\}\to \{1,2\} $ and $ \tau\!: \{1,2\}\to \mathbb{Z} $ be as in Lemma \ref{Lem2DimBijection}. For $ k\in \mathbb{Z} $ let $ \sigma(S_{k}) = \{\lambda_{k,1},\lambda_{k,2}\} $. By Lemma \ref{LemInterweavingWeightedShifts},
    \begin{equation}
        \label{ProofFormInterweavingForEigenvectors}
        \lambda_{k,j}U_{\tau(j),0}v_{j}=U_{\tau(j),0}S_{k}v_{j} = T_{\tau(j)+k}U_{\tau(j),0}v_{j}, \qquad j\in\{1,2\}, k\in\mathbb{Z}.
    \end{equation}
    Since $ \langle U_{\tau(j),0}v_{j},w_{\sigma(j)}\rangle\not=0 $, $ U_{\tau(j),0}v_{j} $ is an eigenvector of $ T_{\tau(j)+k} $ for every $ k\in \mathbb{Z} $. By the assumption, this implies that for $ j\in \{1,2\} $ there exists $ \mu_{j} \in \mathbb{C}\setminus\{0\} $ satisfying $ U_{\tau(j),0}v_{j} = \mu_{j}w_{\sigma(j)} $; since $ U $ is unitary, we have $ \lvert \mu_{j}\rvert \le 1 $. For $ i\in \tau(\{1,2\}) $ define $ V_{i,0}\in \mathbf{B}(H) $ as follows:
    \begin{equation*}
        V_{i,0} x = \sum_{\substack{\ell\in \{1,2\}\\ i = \tau(\ell)}}\langle x,v_{\ell}\rangle w_{\sigma(\ell)}, \qquad x\in H;
    \end{equation*}
    for $ i\notin \tau(\{1,2\}) $ we set $ V_{i,0} = 0 $. It can be easily seen that $ V_{i,0} $ is a partial isometry for every $ i\in \mathbb{Z} $ and that
    \begin{equation*}
        V_{\tau(j),0}^{\ast} x = \sum_{\substack{\ell\in \{1,2\}\\ \tau(j) = \tau(\ell)}}\langle x,w_{\sigma(\ell)}\rangle v_{\ell}, \qquad x\in H, \ j\in\{1,2\}.
    \end{equation*}
    From the above we deduce that for every $ j\in\{1,2\} $, $ V_{\tau(j),0}^{\ast}V_{\tau(j),0} $ is the orthogonal projection of $ H $ onto the space $ \lin\{ v_{\ell}\!: \tau(j) = \tau(\ell) \} $ and that $ V_{\tau(j),0}V_{\tau(j),0}^{\ast} $ is the orthogonal projection of $ H $ onto the space $ \lin\{ w_{\ell}\!: \tau(j) = \tau(\ell) \} $; in particular, if $ \tau(1) = \tau(2) $, then $ V_{\tau(1),0} $ is unitary.
    Define $ V_{0}\in \mathbf{B}(H,\ell^{2}(\mathbb{Z},H)) $ by the formula
    \begin{equation*}
        V_{0}x = (V_{i,0}x)_{i\in \mathbb{Z}}, \quad x\in H.
    \end{equation*}
    If $ x\in H $, then
    \begin{equation*}
        \lVert V_{0}x\rVert^{2} = \sum_{i\in \mathbb{Z}} \lVert V_{i,0}x\rVert^{2} = \sum_{j=1}^{2}\lvert \langle x,v_{j}\rangle\rvert^{2} = \lVert x\rVert^{2}.
    \end{equation*}
    Thus, $ V_{0} $ is the isometry. We will show that $ V_{0} $ satisfies the conditions in Corollary \ref{CorUnitaryEquivalenceShiftsWithDoublyCommutingWeights}. First, we will prove \eqref{FormPositiveWeightsDoublyCommuting} and \eqref{FormNegativeWeightsDoublyCommuting}. For every $ j\in \{1,2\} $ and $ k\in \mathbb{Z} $ we have
    \begin{align*}
        \lambda_{k,j}\mu_{j}w_{\sigma(j)} &= \lambda_{k,j}U_{\tau(j),0}v_{j} = U_{\tau(j),0}S_{k}v_{j} \\
        &\stackrel{\eqref{ProofFormInterweavingForEigenvectors}}{=} T_{\tau(j)+k} U_{\tau(j),0}v_{j} = \mu_{j}T_{\tau(j)+k}w_{\sigma(j)},
    \end{align*}
    which implies that $ \lambda_{k,j}w_{\sigma(j)} = T_{\tau(j)+k}w_{\sigma(j)} $. By the definition of $ V_{\tau(j)} $, this equality takes the form $ V_{\tau(j),0}S_{k}v_{j} = T_{\tau(j)+k}V_{\tau(j),0}v_{j} $ for $ j\in \{1,2\} $ and $ k\in \mathbb{Z} $. If $ j,\ell\in \{1,2\} $ and $ \tau(j)\not=\tau(\ell) $, then for every $ k\in \mathbb{Z} $ we have $ V_{\tau(j),0}S_{k}v_{\ell} = 0 = T_{\tau(j)+k}V_{\tau(j),0}v_{\ell} $. In turn, if $ i\notin \tau(\{1,2\}) $, then $ V_{i,0} = 0 $. Therefore, $ V_{0}S_{k} = D[(T_{i+k})_{i\in \mathbb{Z}}]V_{0} $ for every $ k\in \mathbb{Z} $. From this we obtain that \eqref{FormPositiveWeightsDoublyCommuting} and \eqref{FormNegativeWeightsDoublyCommuting} hold. Next, we will show \eqref{FormWanderingProperty}. If $ k,m\in \mathbb{N}_{1} $, $ k\not=m $, then, by Lemma \ref{LemBasicPropertiesOfOperatorShifts}, for $ j_{1},j_{2}\in \{1,2\} $,
    \begin{align}
        \label{ProofFormWanderingPropertyForPositiveOnBasis}
        &\langle T^{[k]}U_{0}v_{j_{1}},T^{[m]}U_{0}v_{j_{2}}\rangle = \langle T^{[k]}(U_{\tau(j_{1})}v_{j_{1}})^{\tau(j_{1})}, T^{[m]}(U_{\tau(j_{2})}v_{j_{2}})^{\tau(j_{2})}\rangle \\
        \notag&= \langle (T_{\tau(j_{1})+k}\cdots T_{\tau(j_{1})+1}w_{\sigma(j_{1})})^{\tau(j_{1})+k}, (T_{\tau(j_{2})+m}\cdots T_{\tau(j_{2})+1}w_{\sigma(j_{2})})^{\tau(j_{2})+m}\rangle.
    \end{align}
    If $ \tau(j_{1})+k\not=\tau(j_{2})+m $, then the right hand side of \eqref{ProofFormWanderingPropertyForPositiveOnBasis} is equal to zero. If $ \tau(j_{1})+k = \tau(j_{2})+m $, then $ \tau(j_{1})\not=\tau(j_{2}) $, because $ k\not=m $. This implies that $ j_{1}\not = j_{2} $ and $ \sigma(j_{1}) \not= \sigma(j_{2}) $, which implies that $ \langle w_{\sigma(j_{1})},w_{\sigma(j_{2})}\rangle = 0 $. From the fact that $ w_{\sigma(j)} $ is the eigenvector of $ T_{k} $ for every $ j\in \{1,2\} $ and $ k\in \mathbb{Z} $ we deduce that the right hand side of \eqref{ProofFormWanderingPropertyForPositiveOnBasis} is equal to zero also in this case. Since $ \{v_{1},v_{2}\} $ is the orthonormal basis of $ H $, it follows that
    \begin{equation*}
        \langle T^{[k]}U_{0}x,T^{[m]}U_{0}y\rangle = 0, \qquad x,y\in H, k,m\in \mathbb{N}_{1}, k\not=m.
    \end{equation*}
    Similar argument can be used to prove that the above equality holds also when $ k,m\in \mathbb{Z} $ are both non-positive or are of different sign. Therefore, \eqref{FormWanderingProperty} is satisfied. Finally, we will prove \eqref{FormClosedSpan}. Note that it is enough to show that
    \begin{equation*}
        x^{(n)} \in \bigvee_{k\in \mathbb{Z}} T^{[k]}\mathcal{R}(U_{0}), \qquad x\in H, \ n\in \mathbb{Z}.
    \end{equation*}
    Observe that if $ i-\tau(j) < 0 $, then
    \begin{align*}
        \notag T^{[i-\tau(j)]}(V_{\tau(j),0}v_{j})^{\tau(j)} &= T^{\ast (\tau(j)-i)}(V_{\tau(j),0}v_{j})^{(\tau(j))}\\
        \notag&\stackrel{\text{Lemma } \ref{LemBasicPropertiesOfOperatorShifts}}{=} (T_{i+1}\cdots T_{\tau(j)}w_{\sigma(j)})^{(i)}\\
        &\stackrel{\eqref{ProofFormInterweavingForEigenvectors}}{=} (\lambda_{i-\tau(j)+1,j}\cdots \lambda_{0,j}w_{\sigma(j)})^{(i)},
    \end{align*}
    so $ \lVert T^{[i-\tau(j)]}V_{0}v_{j}\rVert = \lambda_{i-\tau(j)+1,j}\cdots \lambda_{0,j} $. Similarly, if $ i-\tau(j)>0 $, then
    \begin{align*}
        \notag T^{[i-\tau(j)]}(V_{\tau(j),0}v_{j})^{(\tau(j))} &= (T_{i}\cdots T_{\tau(j)+1}w_{\sigma(j)})^{(i)} \\
        &= (\lambda_{i-\tau(j),j}\cdots \lambda_{1,j}w_{\sigma(j)})^{(i)},
    \end{align*}
    so $ \lVert T^{[i-\tau(j)]}V_{0}v_{j}\rVert = \lambda_{i-\tau(j),j}\cdots \lambda_{1,j} $. Therefore, if $ x\in H $ and $ n\in \mathbb{Z} $, then
    \begin{equation*}
        x^{(n)} = \sum_{j=1}^{2}\frac{1}{\lVert T^{[n-\tau(j)]}V_{0}v_{j}\rVert}\langle x,w_{\sigma(j)}\rangle T^{[n-\tau(j)]}V_{0}v_{j} \in \bigvee_{k\in \mathbb{Z}} T^{[k]}\mathcal{R}(U_{0}).
    \end{equation*}
    Thus, we get \eqref{FormClosedSpan}. The application of Corollary \ref{CorUnitaryEquivalenceShiftsWithDoublyCommutingWeights}.(ii) and Remark \ref{RemFormulaForUnitaryOperator} gives us (ii).
\end{proof}
In \cite[Example 3.1]{Kos1} the author presented an example of two unitarily equivalent bilateral shifts with operator weights defined on $ \mathbb{C}^{2} $, for which the unitary equivalence cannot be given by a unitary operator of diagonal form. Below we present a general construction of two unitarily equivalent shifts with weights defined on $ \mathbb{C}^{k} $ such that the every unitary operator making them unitarily equivalent has at least $ k $ non-zero diagonals.
\begin{example}
    Assume $ k\in \mathbb{N}_{2} $. For every $ n\in \mathbb{Z} $ let $ (x_{n,i})_{i=1}^{k} \subset (0,\infty) $ be the sequence of positive numbers such that $ x_{k,i}\not= x_{\ell,j} $ for $ (k,i)\not=(\ell,j) $ and that
    \begin{equation}
        \label{ExampleFormUniformBoundedness}
        \sup_{n\in \mathbb{Z}}\max_{i=1,\ldots,k} x_{n,i} < \infty
    \end{equation} 
    Define
    \begin{equation*}
        S_{n} = \begin{bmatrix}
            x_{n,1} & 0 & \ldots & 0\\
            0 & x_{n,2} & \ldots & 0\\
            \vdots & \ddots & \ddots & \vdots\\
            0 & 0 & \ldots & x_{n,k}
        \end{bmatrix}, \qquad n\in \mathbb{Z},
    \end{equation*}
    and
    \begin{equation*}
        T_{n} = \begin{bmatrix}
            x_{n-1,1} & 0 & \ldots & 0\\
            0 & x_{n-2,2} & \ldots & 0\\
            \vdots & \ddots & \ddots & \vdots\\
            0 & 0 & \ldots & x_{n-k,k}
        \end{bmatrix}, \qquad n\in \mathbb{Z}.
    \end{equation*}
    Obviously, both $ S_{n} $ and $ T_{n} $ are positive for every $ n\in \mathbb{Z} $. By \eqref{ExampleFormUniformBoundedness}, the sequences $ (S_{n})_{n\in \mathbb{Z}}, (T_{n})_{n\in \mathbb{Z}}\subset \mathbf{B}(\mathbb{C}^{k}) $ are uniformly bounded. Let $ S, T\in \mathbf{B}(\ell^{2}(\mathbb{Z},\mathbb{C}^{k})) $ be the bilateral weighted shifts with weights $ (S_{n})_{n\in \mathbb{Z}} $ and $ (T_{n})_{n\in \mathbb{Z}} $, respectively. We will show that $ S $ and $ T $ are unitarily equivalent by a unitary operator with at least $ k $ non-zero diagonals. For $ i = 1,\ldots,k $ let $ U_{i,0} $ be the orthogonal projection of $ \mathbb{C}^{k} $ onto $ \lin\{e_{i}\} $, where $ (e_{j})_{j=1}^{k}\subset \mathbb{C}^{k} $ stands for the standard orthonormal basis of $ \mathbb{C}^{k} $; for $ i\in \mathbb{Z}\setminus\{1,\ldots,k\} $ set $ U_{i,0} = 0 $. Define $ U_{0}\in \mathbf{B}(\mathbb{C}^{k},\ell^{2}(\mathbb{Z},\mathbb{C}^{k})) $ by the formula
    \begin{equation*}
        U_{0}x = (U_{i,0}x)_{i\in \mathbb{Z}}, \qquad x\in \mathbb{C}^{k}.
    \end{equation*}
    We will check that $ U_{0} $ satisfies Corollary \ref{CorUnitaryEquivalenceShiftsWithDoublyCommutingWeights}.(ii). Obviously, $ U_{0} $ is an isometry. We will show that \eqref{FormPositiveWeightsDoublyCommuting} and \eqref{FormNegativeWeightsDoublyCommuting} hold. Since $ S_{n} $ and $ T_{n} $ are positive for $ n\in \mathbb{Z} $ it is enough to verify that
    \begin{equation}
        \label{ExampleFormWeightsPositiveInterweaving}
        U_{0}S_{n} = D[(T_{i+n})_{i\in \mathbb{Z}}]U_{0}, \qquad n\in \mathbb{Z}.
    \end{equation}
    For every $ n\in \mathbb{Z} $ and $ j = 1,\ldots,k $ we have
    \begin{equation*}
        U_{0}S_{n}e_{j} = U_{0}x_{n,j}e_{j} = (x_{n,j}e_{j})^{(j)}
    \end{equation*}
    and
    \begin{equation*}
        D[(T_{i+n})_{i\in \mathbb{Z}}]U_{0}e_{j} = D[(T_{i+n})_{i\in \mathbb{Z}}]e_{j}^{(j)} =  (T_{j+n}e_{j})^{(j)} = (x_{n,j}e_{j})^{(j)}.
    \end{equation*}
    Hence, \eqref{ExampleFormWeightsPositiveInterweaving} holds. Next, we will check \eqref{FormWanderingProperty}. Suppose $ n,m\in \mathbb{N}_{1} $ and $ n\not=m $. We will show that
    \begin{equation}
        \label{ExampleFormWanderingProperty}
        \langle T^{[n]}U_{0}x,T^{[m]}U_{0}y\rangle = 0, \qquad x,y\in \mathbb{C}^{k}.
    \end{equation}
    It is enough to verify \eqref{ExampleFormWanderingProperty} for $ x = e_{j_{n}} $, $ y = e_{j_{m}} $, where $ j_{n},j_{m}\in \{1,\ldots,k\} $. By Lemma \ref{LemBasicPropertiesOfOperatorShifts}, for every $ \ell\in \mathbb{N}_{1} $ and $ j = 1,\ldots,k $,
    \begin{equation*}
        T^{\ell}U_{0}e_{j} = T^{\ell}e_{j}^{(j)} = (T_{j+n}\cdots T_{j+1}e_{j})^{j+\ell} = (x_{\ell,j}\cdots x_{1,j}e_{j})^{j+\ell}.
    \end{equation*}
    Hence, $ \langle T^{n}U_{0}e_{j_{n}},T^{m}U_{0}e_{j_{m}}\rangle = 0 $ if $ j_{n}+n\not= j_{m}+m $. If $ j_{n}+n = j_{m}+m $, then $ j_{n}\not=j_{m} $, because $ n\not=m $. Thus, $ \langle T^{n}U_{0}e_{j_{n}},T^{m}U_{0}e_{j_{m}}\rangle = 0 $ also when $ j_{n}+n= j_{m}+m $. Similarly, we can show that \eqref{ExampleFormWanderingProperty} holds for $ m,n $ being both non-positive or of different sign. It remains to verify \eqref{FormClosedSpan}. It is enough to check that
    \begin{equation*}
        x^{(n)} \in \bigvee_{\ell\in \mathbb{Z}} T^{[\ell]}\mathcal{R}(U_{0}), \qquad x\in \mathbb{C}^{k},n\in \mathbb{Z}.
    \end{equation*}
    It is a matter of routine to verify that for $ x\in \mathbb{C}^{k} $ and $ n\in \mathbb{Z} $,
    \begin{equation*}
        x^{(n)} = \sum_{j=1}^{k}\langle x,e_{j}\rangle e_{j} = \sum_{j=1}^{k}\langle x,e_{j}\rangle \frac{1}{\lVert T^{[n-j]}U_{0}e_{j}\rVert^{2}} T^{[n-j]}U_{0}e_{j} \in \bigvee_{\ell\in \mathbb{Z}} T^{[\ell]}\mathcal{R}(U_{0})
    \end{equation*}
    (see the proof of Theorem \ref{Thm2DimUnitaryEquivalenceAtMostTwoNonZeroDiagonals}).
    By Corollary \ref{CorUnitaryEquivalenceShiftsWithDoublyCommutingWeights} and Remark \ref{RemFormulaForUnitaryOperator}, $ S $ and $ T $ are unitarily equivalent by a unitary operator with $ k $ non-zero diagonals. Now we will show that there is no unitary operator with at most $ k-1 $ non-zero diagonals making $ S $ and $ T $ unitarily equivalent. Suppose to the contrary $ U $ is such an operator. By Corollary \ref{CorUnitaryEquivalenceShiftsWithDoublyCommutingWeights}.(i), $ U_{0}\in \mathbf{B}(\mathbb{C}^{k},\ell^{2}(\mathbb{Z},\mathbb{C}^{k})) $ has $ k-1 $ non-zero entries and satisfies \eqref{ExampleFormWeightsPositiveInterweaving}; suppose $ U_{n_{1},0},\ldots,U_{n_{\ell},0}\in \mathbf{B}(\mathbb{C}^{k}) $ are the only non-zero entries of $ U_{0} $, where $ \ell \le k-1 $, $ n_{j} \in \mathbb{Z} $ for $ j = 1,\ldots,\ell $ and $ n_{1}<\ldots < n_{\ell} $. By \eqref{ExampleFormWeightsPositiveInterweaving},
    \begin{equation*}
        x_{n,j}U_{n_{i},0}e_{j} = U_{n_{i},0}S_{n}e_{j} = T_{n_{i}+n}U_{n_{i},0}e_{j}, \qquad i = 1,\ldots,\ell,\ j = 1,\ldots, k.
    \end{equation*}
    Since, $ x_{n,j} $ is the eigenvalue only of the weight $ T_{n+j} $, it follows that
    \begin{equation*}
        U_{n_{i},0}e_{j}\not=0 \implies n_{i} = j, \qquad j= 1,\ldots,k, \ i= 1,\ldots,\ell.
    \end{equation*}
    From the fact that $ U_{0} $ is isometric, we deduce that for every $ j = 1,\ldots,k $ there exists $ n_{i} $ such that $ U_{n_{i},0}e_{j}\not=0 $. Hence, $ \{1,\ldots,k\} = \{n_{1},\ldots,n_{\ell}\} $, which is impossible (we asssumed $ \ell\le k-1 $).
\end{example}
	\bibliographystyle{plain}
	\bibliography{references}

	\noindent Michał Buchała\\  
	michal.buchala@im.uj.edu.pl\\

	\noindent {\small
	\noindent Doctoral School of Exact and Natural Sciences, Jagiellonian University\\
	Łojasiewicza~11\\
	PL-30348 Kraków, Poland\\
	}
	\noindent {\small
	\noindent Institute of Mathematics, Jagiellonian University\\
	Łojasiewicza~6\\
	PL-30348 Kraków, Poland\\
	}\bigskip
\end{document}